\documentclass[12pt]{article}
\usepackage{verbatim}
\usepackage{amssymb}
\usepackage{amsthm}
\usepackage{amsmath}
\usepackage{graphicx}
\usepackage[usenames,dvipsnames] {color}
\usepackage{tikz}
\usepackage{tikz-cd}  
\usepackage{yfonts}
\usepackage{makeidx}
\usepackage{enumerate}
\numberwithin{equation}{section}

\title{Non-commutative manifolds, the free square root and symmetric functions in two non-commuting variables}
\author{Jim Agler
\thanks{Partially supported by National Science Foundation Grant
 DMS 1665260}
\and
John E. M\raise.5ex\hbox{c}Carthy
\thanks{Partially supported by National Science Foundation Grant  
DMS 1565243
}
\and
N. J. Young\thanks{Partially supported
by UK Engineering and Physical Sciences Research Council grants EP/K50340X/1 and  EP/N03242X/1, and 
London Mathematical Society grants 41219 and 41730.
}}

\def\gdel{B_\delta}
\def\ggam{B_\gamma}

\def\d{\mathbb{D}}
\def\D{\mathbb{D}}

\def\a{\mathcal{A}}

\DeclareMathOperator{\ran}{ran}

\def\be{\begin{equation}}
\def\ee{\end{equation}}
\def\norm#1{\| #1 \|}
\def\m{\mathbb{M}}
\def\t{\mathcal{T}}
\def\a{\mathcal{A}}
\def\mn{\mathbb{M}_n}

\def\invn{\mathcal{I}_n}
\def\inv{\mathcal{I}}
\def\c{\mathbb{C}}
\def\d{\mathbb{D}}

\def\calb{\mathcal{B}}

\def\r{\mathbb{R}}
\def\rplus{\mathbb{R}^+}
\def\n{\mathbb{N}}

\def\g{\mathcal{G}}
\def\s{\mathcal{S}}
\def\q{\mathcal{Q}}

\def\calu{\mathcal{U}}
\def\calw{\mathcal{W}}
\def\calr{\mathcal{R}}

\def\ugam{\mathcal{U}_\gamma}
\def\ugamsub#1{\mathcal{U}_{\gamma_{#1}}}
\def\wgam{\mathcal{W}_\gamma}

\def\set#1#2{\{ #1 \, | \, #2\}}

\def\id#1{{\rm id}_#1}

\def\subgt{_{\gamma\tau}}

\def\subgtsub#1{_{\gamma_{#1}\tau_{#1}}}

\DeclareMathOperator{\sep}{sep}
\DeclareMathOperator{\gra}{graph}
\DeclareMathOperator{\diag}{diag}
\newcommand\df{\stackrel{\rm def}{=}}
\newcommand{\C}{\mathbb{C}}
\newcommand\N{\mathbb N}
\newcommand\black{\color{black}}
\definecolor{green}{rgb}{0., 0.6, 0.}

\newcommand\blue{\color{blue}}
\newcommand\red{\color{red}}
\renewcommand\epsilon{\varepsilon}
\newcommand\ga{\gamma}
\newcommand\Ga{\Gamma}
\newcommand\si{\sigma}
\newcommand\ph{\varphi}
\newcommand\half{\tfrac 12}
\newcommand\bbm{\begin{bmatrix}}
\newcommand\ebm{\end{bmatrix}}

\newcommand\nonc{non-commutative }

\newcommand\zf{Zariski-free }
\newcommand\zfly{Zariski-freely }

\newcommand\dut{\rm{f.o.}}

\newcommand{\threepartdef}[6]
{
	\left\{
		\begin{array}{lll}
			#1 & \mbox{if } #2 \\
			  \\
			#3 & \mbox{if } #4 \\
			  \\
			#5 & \mbox{if } #6
		\end{array}
	\right.
}
\newcommand\al{\alpha}
\newcommand\la{\lambda}
\newcommand\de{\delta}

\newcommand\uone{U_1}
\newcommand\vone{V_1}
\newcommand\fone{f_1}
\newcommand\gone{g_1}

\newcommand\hp{\hat \phi}
\newcommand\bda{{ B}_{\delta_\alpha}}
\newcommand\bdb{{B}_{\delta_\beta}}

\newcommand\ta{T_\alpha}
\newcommand\ua{U_\alpha}
\newcommand\ub{U_\beta}

\newcommand\A{{\mathcal A}}
\renewcommand\phi{\varphi}

\def\gg{\textswab{G}}

\def\vare{\varepsilon}

\def\woe{w_1(\vare)}

\def\zoe{z_1(\vare)}
\def\zte{z_2(\vare)}
\newcommand\beq{\begin{equation}}
\newcommand\eeq{\end{equation}}
\newcommand\anc{conditionally\ nc}
\newcommand\Anc{Conditionally\ nc}
\newcommand\1{\mathbf 1}
\newcommand\sd{{\widetilde{D}}}
\newcommand\simf{{\widetilde{f}}}
\newcommand\sym{\mathrm{Sym}}
\newcommand\odd{\mathrm{odd}}
\newcommand\even{\mathrm{even}}
\newcommand\bpm{\begin{pmatrix}}
\newcommand\epm{\end{pmatrix}}
\newcommand\inverse{^{-1}}
\newcommand\fa{\mbox{ for all }}
\makeindex
\date{28th May 2018}
\begin{document}

\bibliographystyle{plain}
\newtheorem{defin}[equation]{Definition}
\newtheorem{lem}[equation]{Lemma}
\newtheorem{prop}[equation]{Proposition}
\newtheorem{thm}[equation]{Theorem}
\newtheorem{claim}[equation]{Claim}
\newtheorem{ques}[equation]{Question}
\newtheorem{fact}[equation]{Fact}
\newtheorem{axiom}[equation]{Technical Axiom}
\newtheorem{newaxiom}[equation]{New Technical Axiom}
\newtheorem{cor}[equation]{Corollary}
\newtheorem{exam}[equation]{Example}
\newtheorem{remark}[equation]{Remark}
\newtheorem{convention}[equation]{Convention}
\maketitle
\begin{abstract}
The richly developed theory of complex manifolds plays important roles in our understanding of  holomorphic functions in several complex variables. It is natural to consider manifolds that will play similar roles in the theory of holomorphic functions in several non-commuting variables. In this paper we introduce the class of \emph{nc-manifolds}, the mathematical objects that at each point possess a neighborhood that has the structure of an \emph{nc-domain} in the \emph{$d$-dimensional nc-universe $\m^d$}. We illustrate the use of such manifolds in free analysis through the construction of the non-commutative Riemann surface for the matricial square root function.
A second illustration is the construction of a non-commutative analog of the elementary symmetric functions in two variables.
 For any symmetric domain in $\m^2$ we construct a  2-dimensional
non-commutative manifold such that the symmetric holomorphic functions on the domain are in bijective correspondence with the holomorphic functions on the manifold.   We also derive a version of the classical Newton-Girard formulae for power sums of two non-commuting variables.
\end{abstract}
\newpage
\tableofcontents
\section{Introduction} \label{introduction}
Free analysis, the study of holomorphic functions in several non-commuting variables, dates back to 1973 and the seminal paper \cite{tay73} by  J. L. Taylor.  The theory has picked up ever greater momentum in the past decade.  
The monograph \cite{kvv14}, by D. S. Kaliuzhnyi-Verbovetskyi and V. Vinnikov contains a panoramic survey of the field up to the time of its writing. Since then, there have been further breakthroughs in both geometry 
(see {\em e.g.} \cite{ddss17,hkm16a,hkm17a,hm12,hm14}
and function theory
(see {\em e.g.} \cite{bmv16a, CPTD17, hptv16}).

Taylor's founding idea for non-commutative analysis was that analytic functions in several non-commuting variables should have the same basic algebraic properties as free polynomials have when viewed as functions on tuples of matrices (of indeterminate size). Here a {\em free polynomial} means a polynomial in non-commuting variables over the field $\c$ of complex numbers.

Traditionally, free analysis has dealt with functions defined on subsets of the {\em $d$-dimensional non-commutative universe} $\m^d$, which comprises the space of $d$-tuples of square matrices, for $d\geq 1$.  Thus
\be\label{2.10}
\m^d = \bigcup_{n=1}^\infty \mn^d
\ee
where $\mn$  denotes the algebra of $n \times n$ matrices over $\c$ and
\[
\mn^d =\set{(x^1,x^2,\ldots,x^d)}{x^r \in \mn \text{ for } r=1,2,\ldots,d}.
\]

In this paper we introduce the notion of an nc- or {\em non-commutative manifold}, which bears the same relation to $\m^d$ as complex manifolds bear to $\C^d$.  This natural extension is needed even for such a basic notion as the ``free Riemann surface'' of the matricial square root function, an object that we construct in Section \ref{SquareRoot}.

We were led to introduce topological nc-manifolds while seeking a non-commutative version of the anciently-known fact that any symmetric holomorphic function of $d$ variables can be expressed as a holomorphic function of the elementary symmetric functions in $d$ variables on a suitable domain.  In particular, if $f(z,w)$ is a holomorphic function on $\c^2$ that is symmetric, in the sense that $f(z,w)=f(w,z)$ for all $z,w\in\c$, then there exists a holomorphic function $g$ on $\c^2$ such that $f(z,w)=g(z+w,zw)$ for all $z,w\in\c$.   We asked whether there is a
 \nonc version of this result.  

Any \nonc  analog must be consistent with results of M. Wolf \cite{wo36}, who studied an algebraic version of this question in 1936.  She proved that the algebra of symmetric free polynomials in $d$ non-commuting variables is not finitely generated when $d>1$, and is in fact isomorphic to the algebra of free polynomials in countably many variables.  It follows that there is no polynomial map $\pi:\m^2\to\m^d$, for any $d\in\n$, with the property that, for every symmetric free polynomial $\ph$ in $2$  variables, there exists a free polynomial $\Phi$ such that $\ph=\Phi\circ\pi$; otherwise, the components of $\pi$ would constitute a finite basis for the algebra of symmetric free polynomials.  Nevertheless, we show that there {\em is} a close parallel to the classical result in the context of free analysis on nc-manifolds.
The main result of the paper is Theorem~\ref{thm9.1}.

Before stating a somewhat special case of the theorem, let us informally describe some key notions
(precise definitions are in Sections~\ref{ncUniverse} and \ref{manifolds}). A subset of $\m^d$ is called an {\em nc set} if it is closed with respect to direct sums. A function $f$  defined on some subset of $\m^d$ is called {\em graded} if $f(x) \in \mn$ whenever $x \in \mn^d$.
A graded function on an nc set is called an {\em nc function} if it preserves direct sums and joint similarities. A graded function $f$ on an arbitrary set $D \subseteq \m^d$ is called {\em \anc}\ if it preserves joint similarities, and, in addition, there is a graded function $\hat{f}$ such that whenever $x$ and $x \oplus y$ are in $D$, then $f(x \oplus y) = f(x) \oplus \hat{f}(y)$.

\begin{thm}
\label{thmaa1}
Let $\s$ be a symmetric nc set in $\m^2$ that is open in the free topology.
There is subset $\s_{oo}$ of $\s$ and a polynomial map $\pi : \m^2 \to \m^3$
such that $\gg_{oo} = \pi(\s_{oo})$ is a topological nc-manifold in the Zariski-free topology, and
such that there is a canonical isomorphism between bounded symmetric nc functions defined on $\s$ and
bounded holomorphic functions on the manifold $\gg_{oo}$ that are \anc.
\be
\label{bigpicaa}
\begin{tikzcd}
  &
 \arrow[swap]{ld}{f_{oo}}  \s_{oo} 
\arrow{r}{\pi} \arrow{d}{\subseteq} & \gg_{oo}   \arrow[r, "F_{oo}"]& \m^1
\\
\m^1 \arrow[r, leftarrow, "f"]& \s &&
\end{tikzcd}
\ee
\end{thm}

The map $\pi$ is given by
\be
\label{eqjma1}
\pi(x^1,x^2) =
 (\tfrac 12(x^1+x^2), \tfrac 14 (x^1-x^2)^2, \tfrac 18 (x^1-x^2)(x^1+x^2)(x^1-x^2)).
 \ee
It is
generically 2-to-1 on $\m^2$, but there is a singular set on which it is many to one. This set is excluded in $\s_{oo}$,
and both $\s_{oo}$ and $\gg_{oo}$ can be given the structure of topological nc-manifolds with respect to a topology
called the Zariski-free topology, defined in Section~\ref{zariski}.
There is also an isomorphism between Zariski-freely holomorphic symmetric functions $f_{oo}$ on $\s_{oo}$ and
Zariski-freely holomorphic functions on $\gg_{oo}$ (with no assumption of boundedness needed)---this is 
Theorem~\ref{thm6.30}.

The bulk of the paper comprises the establishment of a suitable notion of topological nc-manifold and  construction of the manifolds $\gg_{oo}$.
 Topological nc-manifolds are defined in Section \ref{manifolds}. 
The nc-universe $\m^d$ defined by equation \eqref{2.10} is unlike $\r^d$ and $\c^d$ in that it admits numerous natural topologies -- a fact which adds an extra richness to the theory of topological nc-manifolds. 
 
 Section \ref{OneVariable} describes the basic theory of free holomorphic functions of a single variable.  In one variable such functions are determined by their action on scalars, and are given by the holomorphic functional calculus for matrices (Proposition \ref{prop3.40}).

In Section \ref{SquareRoot} we present a simple example of a one-dimensional free manifold -- a non-commutative version of the Riemann surface of the square root function. This theory is an essential component of the construction of the manifold $\g$. 

 In Section \ref{zariski} we define the Zariski-free topology, which is rather subtle, but seems necessary in order to avoid 
 certain singularities in the map $\pi$.
 In Section~\ref{UniversalDomain} we construct a Zariski-free nc-manifold $\g$, which is  the manifold which would
 be $\gg_{oo}$ in Theorem~\ref{thmaa1} if $\s$ were all of $\m^2$.
 Of course bounded symmetric  functions are not of interest when $\s=\m^2$, but there is a version of the theorem
 that does not require boundedness. This is Theorem~\ref{thment}.
\begin{thm}
 There is a canonical bijection between 

 {\rm (i)}
 symmetric nc functions $f$ that are freely holomorphic
 on $\m^2$, and
  
 {\rm (ii)} holomorphic functions $F_{oo}$ defined on the Zariski-free manifold $\g$ that are \anc\ and 
 have the property that for every $w \in \m^2$, there is a free neighborhood $U$ of $w$ such that
 $F_{oo}$ is bounded on $\pi(U) \cap \g$.
 \end{thm}
 
In Sections
 \ref{SymmetricFunctions} and \ref{secj} we prove the main results of the paper, Theorems~\ref{thm9.1}, \ref{thm9.2} and \ref{thm6.30}.
 Finally, in Section~\ref{girard},  we give a non-commutative version of the classical Newton-Girard formulae for power sums in terms of elementary symmetric functions.  There is a significant difference in the non-commutative context: there is no longer a finite algebraic basis for the
algebra of symmetric polynomials.  We can, however, derive  explicit iterative formulae for writing the symmetric sums
 \[
 p_n \ = \ x^n + y^n, 
\]
where  $n \in {\mathbb Z}$,
 as {\em rational} functions composed with the map $\pi$ of equation \eqref{eqjma1}.

\section{Mappings and topologies on $\m^d$} \label{ncUniverse}

\subsection{Nc-functions}
In this section we describe the basic objects of free analysis.  Firstly, the {\em nc-sets} in the nc-universe $\m^d$ of equation \eqref{2.10} are the sets that are closed under direct sums.  They are are the natural domains of definition of {\em nc-functions}, which  are the $\m^1$-valued functions that are graded, preserve direct sums and respect similarity transformations.  We now explain the precise meaning of these terms.

Let $\N$ denote the set of positive integers and, for $n\in\n$,
let
$$
\invn = \set{M \in \mn}{M \text{ is invertible}}.
$$
\index{$\invn$}
For $x_1 = (x_1^1,\ldots,x_1^d) \in \m_{n_1}^d$ and $x_2 = (x_2^1,\ldots,x_2^d) \in \m_{n_2}^d$, we define $x_1 \oplus x_2 \in \m_{n_1+n_2}^d$  by identifying $\c^{n_1} \oplus \c^{n_2}$ with $\c^{n_1 + n_2}$ and direct summing $x_1$ and $x_2$ componentwise, that is,
$$x_1 \oplus x_2 = \Big(x_1^1 \oplus x_2^1, \ldots, x_1^d \oplus x_2^d\Big).$$
Likewise, if $x = (x^1,\ldots,x^d)\in \m_n^d$ and $S \in \m_n$ is invertible, we define $S^{-1} x  S\in \m_n^d$ by
$$
S^{-1} x S = (S^{-1} x^1 S, \ldots, S^{-1} x^d S).
$$
\begin{defin}\label{def2.10}
If $D \subseteq \m^d$ we say that $D$ is an \emph{nc-set} if $D$ is closed with respect to the formation of direct sums,
\index{nc!set}
that is:

 For all $n_1,n_2 \in \N$ and all  $x_1 \in D \cap \m_{n_1}^d, \,  x_2 \in D \cap \m_{n_2}^d$,
\be\label{2.20}
x_1 \oplus x_2 \in D \cap \m_{n_1+n_2}^d, \notag
\ee
\end{defin}

\begin{defin}\label{def2.20}
An \emph{nc-function in $d$ variables} is a function $f$ whose domain is an nc-set $D \subseteq \m^d$, whose codomain is $\m^1$, and which satisfies
\index{nc!function}
\begin{enumerate}
\item[\rm (1)] for all $n\in \N$ and  all $x \in D \cap \mn^d$
\be\label{2.50}
 f(x) \in \mn,\notag
\ee
\item[\rm (2)] for all $n_1,n_2 \in\N$ and all $x_1 \in D \cap \m_{n_1}^d $ and $ x_2 \in D \cap \m_{n_2}^d$
\be\label{2.60}
 f(x_1 \oplus x_2) = f(x_1) \oplus f(x_2), \text{ and} \notag
\ee
\item [\rm (3)]  for all $n\in\N$, all $x \in D \cap \mn^d$ and all $s \in \invn$, if  $s^{-1} x s \in D$ then
\be\label{2.70}
f(s^{-1}x s) = s^{-1}f(x)s. \notag
\ee
\end{enumerate}
A function $f$ that has property {\rm (1)} is said to be {\em graded}.
\index{function!graded}
\end{defin}

\begin{defin}\label{def2.30}
Let $D_1 \subseteq \m^{d_1}$ and $D_2 \subseteq \m^{d_2}$ be nc-sets and let $F:D_1 \to D_2$. We say that $F$ is an \emph{nc-mapping}
\index{nc!mapping}
 if there exist nc-functions $f_1,f_2,\ldots,f_{d_2}$ on $D_1$ such that
\[
F(x) = (f_1(x),f_2(x),\ldots,f_{d_2}(x))
\]
for all $x \in D_1$. If, in addition, $F$ is a bijection and both $F$ and $F^{-1}$ are nc-mappings then we say that $F$ is an \emph{nc-isomorphism}.
\end{defin}
 J. Taylor proved \cite{tay70a} that properties (2) and (3) of Definition~\ref{def2.20} are equivalent to the single property
that $f$ respects intertwinings, in the following sense:

 (4)   {\em for all $n_1,n_2 \in\N$, for $x_1 \in D \cap \m_{n_1}^d $, for $ x_2 \in D \cap \m_{n_2}^d$ and all
 $n_2$-by-$n_1$ matrices $L$ satisfying $L x_1^j = x_2^j L$ for $j=1, \dots, d$, 
\be\label{2.607}
L  f(x_1 )=  f(x_2) L.
\ee
}

\subsection{Topologies on the nc-universe}
\label{subsectopnc}

For each $n$, $\mn$ carries the natural topology induced by the $n \times n$ matrix norm. This topology gives rise to the product topology on $\mn^d$. Using equation \eqref{2.10}, we can endow $\m^d$ with a topology whereby
a set $D\subseteq \m^d$ is open if $D \cap \mn^d$ is open in $\mn^d$ for each $n\in\N$.
This topology has been called the \emph{finitely open topology,} the \emph{disjoint union topology}
and the \emph{coproduct topology}. 
We shall call it the finitely open topology, abbreviated to $\dut$
\index{$\dut$}
\index{topology!disjoint union}
\index{topology!finitely open}
\begin{defin}\label{def2.40}
 Let $\tau$ be a topology on $\m^d$.  
 We say that $\tau$ is an  \emph{nc topology} or equivalently an \emph{admissible topology}
\index{nc!topology}
\index{topology!admissible}
 if $\tau$ has a basis consisting of finitely open nc-sets.
\end{defin}

As we have mentioned, the nc-universe $\m^d$ admits several natural topologies \cite{amif16}.   In this paper we shall consider nc-manifolds  based on $\m^d$ endowed with  three different topologies, the fine, free and Zariski-free topologies. 
The {\em fine} topology is the topology generated by {\em all}  finitely open nc-sets.
\index{topology!fine}

\begin{defin}
A domain in $\m^d$ is a finely open set.
\end{defin}
\index{domain}
Domains do not need to be nc-sets.   For example the set 
\be
\label{eqb1}
\mathcal{Q}=\{x\in \m^1: \sigma(x)\cap\sigma(-x)=\emptyset\}
\ee
 is a  domain (see Proposition \ref{prop4.10} below), indeed a free domain (see Definition \ref{def2.50}),  but is not an nc-set.

We now describe the free topology. It was introduced in \cite{amfree} in the context of a non-commutative Oka-Weil approximation theorem.  

 For $\delta = [\delta_{ij}]$ an $I\times J$ matrix of free polynomials, let $B_\delta \subseteq \m^d$ be defined by
\[
\gdel = \set{x\in \m^d}{\norm{[\delta_{ij}(x)]} < 1}.
\]
Observe that if $\delta_1$ and $\delta_2$ are matrices of free polynomials, then
\[
B_{\delta_1} \cap B_{\delta_2} = B_{\delta_1 \oplus \delta_2}.
\]
As a consequence of this fact, the collection of sets of the form $B_\delta$ is closed with respect to finite intersections and thus forms a basis for a topology.
\begin{defin}\label{def2.50}
The \emph{free topology on $\m^d$} is the topology  for which a basis is the collection of sets of the form $\gdel$ where $\delta$ is a matrix of free polynomials in $d$ variables. A set that is open in the free topology is a \emph{free domain}.
\end{defin}
\index{topology!free}
Since $B_\delta$ is an nc-set and is open in the \dut\ topology, the free topology is admissible. 

 The free topology is much coarser than the \dut\ topology, so much so that while the \dut\ topology is Hausdorff,  the topology induced on $\mn^d$ by the free topology is not even $T_1$ for any $n\geq 2$ and $d\geq 1$.
One way to see this is to observe that free open sets are invariant under unitary conjugation.

\begin{prop}\label{frconn}
Any freely open nc-set is freely path-connected.
\end{prop}
\begin{proof}
Let $U$ be a freely open nc-set in $\m^d$.
  Consider $x, \ y\in U$; then $x\oplus y \in U$. Define $h: [0,1] \to U$ by
\[
h(t) = \threepartdef{x}{0\leq t< \half}{x\oplus y}{t=\half}{y}{\half < t\leq 1.}
\]
$h$ is constant on the intervals $[0,\half)$ and $(\half, 1]$, and so freely continuous on those intervals.  $h$ is also freely continuous at the point $\half$.  For consider any basic free neighborhood $B_\ga$ of the point $h(\half)=x\oplus y$ in $U$.  Then we have $\|\ga(x\oplus y)\| < 1$, and so 
$\|\ga(x)\|<1$ and $\|\ga(y)\| <1$, which is to say that $h(t)\in B_\ga$ for all $t$ in the  neighborhood $[0,1]$ of $\half$.  Thus $h$ is continuous at $\half$, and so $h$ is a continuous path in $U$ such that $h(0)=x$ and $h(1)=y$.
\end{proof}
\begin{cor}\label{mdconn}
$\m^d$ is connected and locally connected in the free topology.
\end{cor}
\begin{proof}
Path-connectedness implies connectedness.  To say that $\m^d$ is locally connected means that every point in $\m^d$ has a neighborhood base of connected sets.  The sets $\gdel$ comprise such a base.
\end{proof}
\begin{remark}\label{slightlymore} \rm
The proof of Proposition \ref{frconn} shows slightly more.  For any freely open set $U$ and any points $x,y\in U$ such that $x\oplus y \in U$, there is a freely continuous path in $U$ from $x$ to $y$.
\end{remark}
\black

\section{Holomorphy with respect to admissible topologies} \label{holomorphic}
  The notion of holomorphy for a function on $\m^d$ depends on the chosen topology of $\m^d$. 
\begin{defin}\label{def2.60}
Let $\tau$ be an admissible topology on $\m^d$. We say that a graded function $f$ defined on a set $D \subseteq \m^d$ is a \emph{$\tau$-holomorphic function} if
\index{function!$\tau$-holomorphic}
\begin{enumerate}
\item [\rm (1)] $D$ is open in the topology $\tau$;
\item [\rm (2)] $f$ is \emph{$\tau$-locally nc}, that is, for each $x \in D$ there exists an nc set $U \in \tau$ such that $x\in U \subseteq D$ and $f|U$ is an nc-function;
\index{function!$\tau$-locally nc}
\item [\rm (3)] $f$ is \emph{$\tau$-locally bounded}, that is, for each $x \in D$, there exists a $ \tau$-neighborhood $U$ of $x$ such that $f|U$ is bounded.
\index{function!$\tau$-locally bounded}
\end{enumerate}
\end{defin}
At first sight this is a surprising definition.  It is at least partially justified by its relation to the following notion of analyticity.
\begin{defin}\label{defanalytic}
A graded function $f: D\to \m^1$, where $D\subseteq \m^d$, is {\em analytic} if $D$ is finitely open and, for every positive integer $n$,
the restriction of $f$ to $D\cap \m^d_n$ is analytic in the usual sense of several complex variables.
\end{defin}
\index{analytic}
\index{function!analytic}

\begin{prop}\label{analytic}
For any admissible  topology $\tau$ on $\m^d$, every $\tau$-holomorphic function $f$ on a domain $D$ in $\m^d$ is analytic.
\end{prop}
\begin{proof}
Consider a point $x\in D\cap \mn^d$.  Since $\ph$ is $\tau$-holomorphic we may choose a $\tau$-neighborhood $U$ of $x$ in $D$ on which $\ph$ is an nc-function and another $\tau$-neighborhood $V$ on which $\ph$ is bounded.  Since the topology $\tau$ is admissible, it is an nc-topology and so we may assume that $V$ is an nc-set.  Then $U\cap V$ is an nc-set and is a $\tau$-neighborhood of $x$ in $D$ on which $\ph$ is a bounded nc-function.  Theorem 4.6 of \cite{amfree} asserts that under these hypotheses $\ph$ is analytic on $U\cap V\cap \mn^d$.  Thus $\ph$ is analytic on some neighborhood of an arbitrary point of $D\cap \mn^d$, and therefore $\ph$ is analytic on $D\cap \mn^d$.
\end{proof}

\blue
We shall say that an nc set $U$ is hereditary if whenever $x \oplus y$ is in $U$ then $x$ and $y$ are in $U$.
An admissible hereditary topology is one that has a basis consisting of finitely open hereditary nc-sets.
\black

\begin{prop} \label{holnc}
Let $\tau$ be an admissible herditary topology on $\m^d$, let $D$ be a $\tau$-open set and let $f$ be a $\tau$-holomorphic function on $D$.
If $D$ is an nc set then $f$ is an nc function.
\end{prop}
\begin{proof}
Consider $z,w\in D$.  Since $z\oplus w \in D$, $f(z\oplus w)$ is defined.
By Definition \ref{def2.60}(2) there is a hereditary  nc set $U \in\tau$ such that $z\oplus w \in U \subseteq D$ and $f\left| U\right.$ is an nc function.
It follows that $f(z\oplus w)=f(z)\oplus f(w)$.
\end{proof}

The following statement is routine to check.
\begin{prop}\label{itsanalgebra}
Let $\tau$ be an admissible topology on $\m^d$ and let $D$ be a $\tau$-open set.  The set of $\tau$-holomorphic functions on $D$ is an algebra with respect to pointwise operations.   The set $H^\infty_\tau(D)$ of bounded $\tau$-holomorphic functions on $D$ is a Banach algebra with respect to pointwise operations and the supremum norm
\[
 \|f\|_\infty = \sup_{x\in D} \|f(x)\|.
\]
\end{prop}
\black
The case where $\tau$ is the free topology will be of special interest here, and in this case we refer to $\tau$-holomorphic functions as \emph{free} (or {\em freely})    \emph {holomorphic functions}. 
\index{function!freely holomorphic}
Such functions are particularly well behaved on account of the following theorem \cite{amfree} (it is also proved in \cite{bmv16b} and \cite{amfreeII}).
\begin{thm}\label{thm2.10}
Let $D \subseteq \m^d$ be a free domain.  A function  $f:D \subseteq \m^d\to \m^1$ is a free holomorphic function if and only if $f$ can be locally uniformly approximated by free polynomials. That is, $f$ is freely holomorphic if and only if for each $x \in D$ there exists a free domain $U$ satisfying $x \in U \subseteq D$  with the property that for each $\epsilon >0$ there exists a free polynomial $p$ such that
\[
\sup_{y \in U} \norm{f(y)-p(y)} < \epsilon.
\]
\end{thm}

 To prove that freely holomorphic functions are freely continuous we need the following simple observation.  
In the Lemma and elsewhere the norm $\|\cdot\|$ on $\mn^d$ is given by
\[
\|x\| = \max \{ \|x^1\|, \dots,\|x^d\| \}
\]
where $\|x^j\|$ is the standard $C^*$ norm on $\mn$.
\begin{lem}\label{Lip}
A free polynomial is freely locally Lipschitz.  That is, if $p$ is a free polynomial in $d$ variables and $z\in\m^d$ then there exist a free neighborhood $G$ of $z$ and a positive constant $K$ such that
\[
\norm{p(x)-p(y)}\leq  K\norm{x-y}
\]
for all $x, y \in G$.
\end{lem}
\begin{proof}
It is enough to prove the statement in the case that $p$ is the monomial
\[
p(x)= x^{r_1}x^{r_2}\dots x^{r_k}
\]
for some $k\in \n$ and $ r_1,\dots, r_k \in \{1,\dots, d\}$.  Let 
\[
\delta(x)= (1+\|z\|)^{-1}\diag\{x^1,\dots,x^d\}.
\]
Then $x\in B_\delta$ if and only if $\|x^j\| < 1+\|z\|$ for $j=1,\dots,d$.  Thus $z\in B_\delta$.  We have,
for any $x,y \in\m^d$,
\begin{align*}
p(x)-p(y)&=(x^{r_1}-y^{r_1})x^{r_2}\dots x^{r_k}+ y^{r_1}(x^{r_2}-y^{r_2}) x^{r_3}\dots x^{r_k} +\\
	& \hspace{1cm} \dots + y^{r_1}\dots y^{r_{k-1}}(x^{r_k} - y^{r_k}).
\end{align*}
Hence, for every $x,y\in B_\delta$,
\[
\|p(x)-p(y)\| \leq k(1+\|z\|)^{k-1} \|x-y\|.
\]
Thus $p$ is freely locally Lipschitz.
\end{proof}
\begin{prop}\label{prop2.30}
Let $U$ be a free domain in $\m^d$.  A function $f:U \to \m^1$ is  freely holomorphic if and only if $f$ is a freely locally nc-function (as in Definition {\rm \ref{def2.60} (2)}) and $f$ is continuous when $U$ and $\m^1$ are equipped with the free topologies.
\end{prop}
\begin{proof}
First assume that   $U$ is an nc-set,  $f$ is an nc-function and $f$ is continuous when $D$ and $\m^1$ are equipped with the free topologies.   
In the light of Definition \ref{2.60}  it suffices to show that $f$ is locally bounded. Fix $z \in U$. As $f$ is assumed continuous and $U$ is open in the free topology, there exists a free  matricial  polynomial $\delta$ such that $z \in B_\delta \subseteq U$ and such that
\[
f(B_\delta) \subseteq \set{x\in \m^1}{ \norm{x} < \norm{f(z)} + 1}.
\]
This proves that $f$ is locally bounded.

Now assume that $U$ is a free domain in $\m^d$ and $f$ is a free holomorphic function on $U$. The continuity of $f$ will follow if we can show that $f^{-1} (B_q)$ is freely open in $\m^d$ whenever $q$ is a square matrix of polynomials in one variable.

Fix $z \in f^{-1} (B_q)$. By Theorem \ref{thm2.10} there exists a square matrix $\delta$ of free polynomials such that $f$ is bounded on $B_\delta$ and a sequence of free polynomials $p_1,p_2,\ldots$ such that $z \in B_\delta \subseteq U$ and
\be\label{}
\lim_{k \to \infty}\sup_{x \in B_\delta}\norm{f(x) - p_k(x)} = 0.
\ee
It follows that the polynomials $p_k$ are uniformly bounded on $B_\delta$.
By Lemma \ref{Lip} $q$ is locally Lipschitz on $U$, and hence
\be\label{3.20a}
\lim_{k \to \infty}\sup_{x \in B_\delta}\norm{q(f(x)) - q(p_k(x))} = 0.
\ee
Fix a strictly decreasing sequence $t_1,t_2,\ldots$ with $t_k >1$ for all $k$ and $t_k \to 1$ as $k \to \infty$. Use  the property \eqref{3.20a} to construct inductively a sequence $k_1,k_2,\ldots$ of positive integers such that
\begin{align*}
B_{t_1(q\circ f)} \cap B_\delta
&\subseteq B_{t_2 (q \circ p_{k_1})}\cap B_\delta\\
 &\subseteq B_{t_3(q\circ f)} \cap B_\delta \\
 &\subseteq B_{t_4(q \circ p_{k_2})}\cap B_\delta\\
  &\subseteq B_{t_5(q\circ f)} \cap B_\delta\\
  & \ldots
\end{align*}
The construction ensures that
\[
B_{q\circ f} \cap B_\delta= \bigcup_{i=1}^\infty (B_{t_{2i-1} (q\circ f)} \cap B_\delta)= \bigcup_{\ell=1}^\infty (B_{t_{2\ell}(q \circ p_{k_\ell})} \cap B_\delta).
\]
The last union in this formula is a union of basic freely open sets and so is a freely open set in $\m^d$.   Hence $B_{q\circ f} \cap B_\delta$ is freely open.  Since $B_{q\circ f}=f^{-1}(B_q)$ we have $z \in B_{q\circ f} \cap B_\delta \subseteq f^{-1}(B_q)$ and therefore $f^{-1}(B_q)$ is open in $\m^d$.
\end{proof}
Similar results on analyticity of nc functions are known in several contexts, for example \cite[Chapter 7]{kvv14}.
\begin{defin}\label{def2.70}
Let $\tau$ be an admissible topology on $\m^{d_1}$ and let $D \subseteq \m^{d_1}$ be a $\tau$-open set.
We say that $F:D\to \m^{d_2}$ is a \emph{$\tau$-holomorphic mapping} if there exist $\tau$-holomorphic functions $f_1,f_2,\ldots,f_{d_2}$   on $D$ \black such that
\[
F(x) = (f_1(x),f_2(x),\ldots,f_{d_2}(x))
\]
for all $x \in D$. 

Let $D_1,D_2$ be $\tau$-open sets in $\m^d$.  We say that $F:D_1 \to D_2$ is a 
\emph{$\tau$-biholomorphic mapping} if $F$ is a bijection and both $F$ and $F^{-1}$ are $\tau$-holomorphic mappings.
\end{defin}
\begin{prop}\label{prop2.40}
Let $F=(f_1,\ldots,f_{d})$ be a freely locally nc mapping defined on a free open set $D \subseteq \m^{d}$. The following  statements are equivalent.
\begin{enumerate}
\item [\rm (1)] $F$ is a freely holomorphic map; 
\item [\rm (2)] $F$ is freely locally bounded;
\item [\rm (3)] $F$ is freely continuous.
\end{enumerate}
In particular, $F$ is a freely biholomorphic map if and only if $F$ is a homeomorphism in the free topology.
\end{prop}
This result is proved in much the same way as  Proposition \ref{prop2.30}.

\section{Nc-manifolds and topological nc-manifolds} \label{manifolds}

\subsection{Nc-manifolds}

In this section we  define nc-manifolds, topological nc-manifolds and free manifolds.
The concept of an nc-manifold is the generalization of the concept of an nc-set as defined in Definition \ref{def2.10} and as such is purely algebraic in nature.
The notions of charts, atlases and transition functions transfer directly from the classical theory, only in the new context a chart is a bijective map from a subset of an nc-manifold to a subset of $\m^d$; in the case of a {\em topological} nc-manifold the image of a chart is an open set in $\m^d$ with respect to some specified topology.

\begin{defin}\label{defn2.30}
If $X$ is a set, then we say that $\alpha$ is a \emph{$d$-dimensional co-ordinate patch} or {\em chart}
\index{chart}
\index{co-ordinate patch}
 on $X$ if $\alpha$ is a bijection from a set $U_\alpha \subseteq X$ to a set $D_\alpha \subseteq \m^d$.  If $\alpha$ and $\beta$ are a pair of $d$-dimensional co-ordinate patches on $X$ with $U_\alpha \cap U_\beta \ne \emptyset$, then we define the \emph{transition map} $T_{\alpha\beta}$ by
\index{map!transition}
\be\label{2.80}
 T_{\alpha\beta}:\alpha(U_\alpha \cap U_\beta) \to \beta(U_\alpha \cap U_\beta),
 \ee
 \be\label{2.90}
  T_{\alpha\beta}(x) = \beta \circ \alpha^{-1} (x),\qquad x \in \alpha(U_\alpha \cap U_\beta).
 \ee

If $X$ is a set  then we say that $\a$ is a \emph{$d$-dimensional nc-atlas for $X$} 
\index{nc!atlas}
if $\a$ is a collection of $d$-dimensional co-ordinate patches on $X$,
\be\label{2.98}
\bigcup_{\alpha \in \a} U_\alpha = X
\ee
and, for all $\alpha,\beta \in \a$, 
\begin{enumerate}
\item [\rm (1)] $\alpha(U_\alpha \cap U_\beta)$ is a union of nc-sets, and
\item [\rm (2)] for every nc-subset $W$ of $\alpha(U_\alpha \cap U_\beta)$, the restriction of $T_{\al\beta}$ to $W$ is an nc-mapping. 
\end{enumerate}

A \emph{$d$-dimensional nc-manifold} is an ordered pair $(X,\a)$ where $X$ is a set and $\a$ is a $d$-dimensional nc-atlas for $X$.
\index{nc!manifold}

If $(X_1,\a_1)$ and $(X_2,\a_2)$ are nc-manifolds of dimensions $d_1, \ d_2$ respectively, then a map $f:X_1 \to X_2$ is an {\em nc-map} (or {\em nc-mapping}) if, for every $\al\in\a_1$  (with domain $U_\al$ and codomain $D_\al$) and $\beta\in\a_2$ (with domain $V_\beta$ and codomain $E_\beta$) 
\begin{enumerate}
\item[\rm  $(1')$]  the set
\[
W_{\al\beta} \df \al(U_\al   \cap f^{-1}(V_\beta))
\]
is a union of nc-sets, and
\item[\rm $(2')$]  for every nc-subset $W$ of $W_{\al\beta}$
\[
\beta \circ f \circ (\al^{-1}| W) : W \to E_\beta  \subset \m^{d_2}
\]
is an nc-map.
\end{enumerate}

If  $(X_1,\a_1)$ and $(X_2,\a_2)$ are  nc-manifolds then an {\em nc-isomorphism} from $X_1$ to $X_2$ is a bijective map $f:X_1\to X_2$ such that both $f$ and $f^{-1}$ are nc-maps.
\end{defin}

\begin{remark}\label{astar} \rm
1.  If $(X,\a)$ is an nc-manifold then, for every $\al\in\a, \, D_\al$ is a union of nc-sets.   This is a consequence of condition (1) in the definition of an nc-atlas above and the identity
\[
D_\al = \bigcup_{\beta\in\a} \al(U_\al \cap U_\beta).
\]

\noindent 2.  If $(X,\a)$ is an nc-manifold then there is an nc-atlas $\a^*$ on $X$ such that the identity map $\id{X}$ is an nc-isomorphism from $(X,\a)$ to $(X,\a^*)$ and, for every $\al\in\a^*$, the range of $\al$ is an nc-set (not merely a {\em union} of nc-sets).  Indeed, we may define $\a^*$ to be the set of all maps $\al|V$ for some $\al\in\a$ and $V\subset U_\al$ such that $\al(V)$ is an nc-set.  It would be possible to develop the theory of nc-manifolds with the assumption that the ranges of charts are always nc-sets, but in the topological context it is convenient to allow them to be merely {\em unions} of nc-sets.
\end{remark}

To prove that $\id{X}:(X,\a) \to (X,\a^*)$ is an nc-isomorphism, consider any $\al\in \a$ and $\ga\in\a^*$, say $\ga= \beta|V$ where $V\subset U_\beta$ and $\beta(V)$ is an nc-subset of $D_\beta$.  Then $W_{\al\ga} = \al(U_\al\cap V)\subset \al(U_\al \cap U_\beta)= T_{\beta\al}(\beta(V))$.  Since $\beta(V)$ is an nc-set and $T_{\beta\al}$ is an nc-map, it follows that $W_{\al\ga}$ is an nc-set.  Moreover, for every nc-subset $W$ of $W_{\al\ga}=\al(U_\al \cap V)$,
\[
\beta\circ \id{X} \circ (\al^{-1}|W) = T_{\al\beta}
\]
is an nc-map.  Similarly, reversing the roles of $\a$  and $\a^*$, we obtain, for any nc-subset $W$ of $W_{\ga\al}=\beta(U_\al\cap V)$
\[
\al\circ \id{X}\circ (\ga^{-1}| W) = T_{\beta\al}| W
\]
is also an nc-map.  Thus $\id{X}$ is an nc-isomorphism with respect to $\a$ and $\a^*$.

\subsection{Topological nc-manifolds}
We now consider the case where $X$, in addition to carrying the structure of an nc-manifold, is a topological space.  $X$ will be a topological nc-manifold if it is locally homeomorphic to an open set in some $\m^d$.   Since there is no one ``correct'' topology to place on $\m^d$,  we first fix a topology $\tau$ on $\m^d$. We assume that $\tau$ is an admissible topology in the sense of  Definition \ref{def2.40}.

If $X$ is a set and $\t$ is a topology on $X$, then we say that $\alpha$ is a \emph{topological $d$-dimensional nc-co-ordinate patch on $(X,\t)$ with respect to $\tau$}  if $\alpha$ is a $d$-dimensional nc-co-ordinate patch on $X$ and, in addition,
\be\label{2.110}
\alpha:U_\alpha \to D_\alpha \text{ is a homeomorphism}.
\ee
Here $U_\alpha \subseteq X$ is a $\t$-open set equipped with the relative topology induced by $\t$ and $D_\alpha \subseteq \m^d$ is a $\tau$-open nc-set equipped with the relative topology induced by $\tau$.   Since $\tau$ is locally nc, $D_\al$ is a union of open nc-sets.

If $(X,\t)$ is a topological space then we say that $\a$ is a \emph{$d$-dimensional topological nc-atlas
\index{nc!atlas!topological}
 for $(X,\t)$ with respect to $\tau$ } if $\a$ is a collection of topological $d$-dimensional nc-co-ordinate patches on $(X,\t)$ and $\a$ is an nc-atlas for $X$  with respect to $\tau$  in the sense of Definition \ref{defn2.30}.

\begin{defin} \label{def2.32}
Let  $\tau$ be an admissible topology on $\m^d$.
A \emph{topological $d$-dimensional nc-manifold  with respect to $\tau$}
\index{nc!manifold!topological}
  is an ordered triple $(X,\t,\a)$ where $X$ is a set, $\t$ is a topology on $X$ and $\a$ is a $d$-dimensional topological nc-atlas for $(X,\t)$  with respect to $\tau$.  

In the special case where $\tau$ is the free  topology we say that $X$ is a {\em free  manifold}.
\end{defin}
\index{nc!manifold!free}
 The manifolds studied in analysis usually have some smoothness property, such as $C^\infty$ or analyticity, whereas the topological nc-manifolds introduced in Definition \ref{def2.32} are not assumed smooth.    It is simple to extend the notion of topological nc-manifold further to bring in appropriate notions of smoothness.
\begin{defin} \label{holman}
Let $\tau$ be an admissible topology on $\m^d$. A {\em $d$-dimensional holomorphic nc-manifold  with respect to $\tau$} is a topological $d$-dimensional nc-manifold $(X,\t,\a)$ with respect to $\tau$ such that, for all $\alpha,\beta\in\a$, the transition map $T_{\alpha\beta}$ is a $\tau$-holomorphic mapping in the sense of Definitions {\rm \ref{def2.70}} and {\rm \ref{def2.60}}.

If $(X,\t,\a)$ is a $d$-dimensional holomorphic nc-manifold  with respect to $\tau$ then a function $F:X\to \m^1$ is said to be {\em holomorphic on $X$} if, for every $\alpha\in\a$, the map
\[
F\circ \alpha^{-1}: \alpha(U_\alpha) \to \m^1
\]
is a $\tau$-holomorphic function on the $\tau$-open set $\alpha(U_\alpha)$.
\end{defin}
More generally, we can define the notion of a holomorphic map between two topological nc-manifolds.
\begin{defin}\label{twotopman}
Let $\tau_1,\tau_2$ be admissible topologies on $\m^{d_1}, \m^{d_2}$ respectively, and let $(X_j, \t_j,\a_j)$ be a topological $d_j$-dimensional manifold with respect to $\tau_j$ for $j=1,2$.  A map $f:X_1 \to X_2$ is said to be {\em holomorphic} at a point $x\in X_1$ if there exist $\al\in\a_1$ (with domain $U_\al$ and domain $D_\al$) and $\beta\in\a_2$ (with domain $V_\beta$ and range $E_\beta$) such that $x\in U_\al, \, f(x) \in V_\beta$ and the map
\[
\beta \circ f\circ \al^{-1}: \al(U_\al\cap f^{-1}(V_\beta)) \to  E_\beta 
\]
is $\tau_1$-holomorphic.  We say that $f$ is holomorphic on $X_1$ if $f$ is holomorphic at every point of $X_1$.
\end{defin}
\index{map!holomorphic}
 Note that the definition is independent of $\tau_2$, save for the requirement that $\tau_2$ be admissible.

In the case that $\tau$ is the free topology,  Proposition \ref{prop2.40} tells us that continuity implies holomorphy.  Accordingly, 
free manifolds are a precise noncommutative analog of complex manifolds. Indeed, they are the topological manifolds that are equipped with an atlas of homeomorphisms onto free domains in $\m^d$ with the property that the transition functions are freely biholomorphic maps. To see this fact, assume that $(X,\t,\a)$ is a free manifold, $\alpha,\beta \in \a$, and $U_\alpha, U_\beta \in \t$  with $U_\alpha \cap U_\beta \neq \emptyset$ . As $U_\alpha \cap U_\beta \in \t$, the hypothesis  \eqref{2.110} implies that $\alpha(U_\alpha \cap U_\beta)$ and $\beta(U_\alpha \cap U_\beta) $ are open in the free topology. The hypothesis \eqref{2.110} also implies that
\[
\beta \circ \alpha^{-1}:\alpha(U_\alpha \cap U_\beta) \to \beta(U_\alpha \cap U_\beta)
\]
is a homeomorphism. Hence, by Proposition \ref{prop2.40}, $\beta \circ \alpha^{-1}$ is a free biholomorphic mapping. Conversely, if the transition functions are assumed to be free biholomorphic mappings, then Proposition \ref{prop2.40} implies that the transition functions are free holomorphic mappings.

If  $(X,\t,\a)$ is a topological $d$-dimensional nc-manifold with respect to $\tau$ then the transition function $T_{\alpha\beta}$ is a composition of two homeomorphisms  and hence is a homeomorphism between $\alpha(U_\alpha\cap U_\beta)$ and $\beta(U_\alpha\cap U_\beta)$ in their respective $\tau$ topologies, as well as being an nc-isomorphism between nc-sets when restricted to any nc-subset of $\alpha(U_\alpha\cap U_\beta)$. 

If  $(X,\t,\a)$ is a topological $d$-dimensional nc-manifold with respect to $\tau$ and $\tau'$ is a finer admissible topology on $\m^d$ then we define the topology $\t'$ on $X$ to be the one for which a base is
\[
\{\al^{-1}(V): \al\in\a, V\subset D_\al\mbox{ and } V\in\tau'\}.
\]
Then $(X,\t',\a)$ is a $d$-dimensional topological nc-manifold with respect to $\tau'$.  The topology $\t'$ is finer than $\t$ and  the map $\id{X}:(X,\t',\a) \to (X,\t,\a)$ is a holomorphic map of topological nc-manifolds.

In particular we may take $\tau'$ to be the
fine  topology on $\m^d$.  Then the topology $\t'$
is the finest topology for which $X$ is a topological nc-manifold.

\section{Free holomorphic functions in one variable} \label{OneVariable}
In this section we show that a free holomorphic function in one variable is determined (via the functional calculus) by its restriction to $\m_1^1$.

Let $\rplus$ denote the set of positive real numbers. If $c \in \c^k$ and $r \in {\rplus}^k$, we define $\Delta(c,r) \subseteq \c$ by
\index{$\Delta(c,r)$}
\begin{align*}
\Delta(c,r) &= \bigcup_{j=1}^k\  \{z \in \c: \,  |z-c_j|<r_j\}\\
&=\bigcup_{j=1}^k\  (c_j +r_j \d)
\end{align*}
and define $D(c,r)\subseteq \m^1$ by
\index{$D(c,r)$}
\[
D(c,r) = \{x \in \m^1: \, \sigma(x) \subseteq \Delta(c,r)\}
\]
where $\si(x)$ denotes the spectrum of the matrix $x$.
In the sequel we make the standing assumption that the radii $r_1,\ldots,r_k$ are so small that the discs $c_1 +r_1 \d,\ldots,c_k +r_k \d$ are pairwise disjoint.
\begin{prop}\label{prop3.10}
$D(c,r)$ is a free domain.
\end{prop}
\begin{proof}
Assume that $M \in D(c,r)$. We construct a polynomial $\delta$ satisfying
\be\label{}
M \in \gdel \subseteq D(c,r). \notag
\ee
Choose a polynomial $q$ satisfying
\[
\sigma(M) \subseteq \set{z}{|q(z)|<1} \subseteq \Delta(c,r).
\]
As $\sigma(q(m)) = q(\sigma(m)) \subseteq \d$,
there exists an integer $N$ such that\\ $\norm{q(M)^N} < 1$. If we set $\delta = q^N$, it follows that $M \in \gdel$.

It remains to show that $\gdel \subseteq D(c,r)$. Fix $x \in \gdel$ such that $\norm{q(x)^N} <1$. It follows that
\[
q(\sigma(x))^N  = \sigma(q(x)^N) \subseteq \d.
\]
This implies that $q(\sigma(x)) \subseteq \d$. But then,
\[
\sigma(x) \subseteq  \set{z}{|q(z)|<1} \subseteq \Delta(c,r),
\]
which implies that $x \in D(c,r)$.
\end{proof}

If $c\in \c^k$ and $r \in {\rplus}^k$ we let $[c,r] =\{[c,r]_1,\ldots,[c,r]_k\}$ denote the system of paths where, for $j=1,\ldots,k$,
\[
[c,r]_j(t) = c_j +r_j e^{it}, \qquad 0 \le t \le 2\pi.
\]
If $f$ is holomorphic on a neighborhood of $\Delta(c,r)$ and $x \in D(c,r)$ then we may employ the Riesz Functional Calculus to define $f^\wedge(x)$ by the formula
\be\label{3.10}
f^\wedge(x) = \int_{[c,s]}  f(z) (z-x)^{-1}\ dz
\ee
where $s\in {\rplus}^k$ is chosen so that $\sigma(x) \subseteq \Delta(c,s)$ and $\Delta(c,s)^- \subseteq \Delta(c,r)$.
\begin{prop}\label{prop3.3}
If $f$ is holomorphic on $\Delta(c,r)$, then $f^\wedge$ is freely holomorphic on $D(c,r)$.
\end{prop}
\begin{proof}
It is straightforward to verify that $f^\wedge$ is an nc-function. The proposition will follow if we can show that $f^\wedge$ is locally bounded. To that end, fix $M \in D(c,r)$. Chose $s \in {\rplus}^k$ so that  $\sigma(M) \subseteq \Delta(c,s)$ and $\Delta(c,s)^- \subseteq \Delta(c,r)$. Choose $C \in \rplus$ such that
\be\label{3.20}
C > \max_{z \in [c,s]} \norm{(z-M)^{-1}}.
\ee
Choose a finite set $S \subseteq [c,s]$ such that, for all $z \in [c,s]$, there exists $w \in S$ such that
\be\label{3.30}
 |z-w|< \frac{1}{2C}.
\ee
Finally, define $G \subseteq \m^1$ by
\be\label{3.40}
G = \bigcap_{w \in S} \set{x \in \m^1}{ w \not\in \sigma(x) \text{ and } \norm{(w - x)^{-1}} < C}.
\ee
As $M \in G\cap D(c,s) \subseteq D(c,r)$, the proof of Proposition \ref{prop3.3} will be complete if we can prove the following two claims.\\ \\
{\bf Claim 1.} $G \cap D(c,s)$ is a free domain.\\ \\
{\bf Claim 2.} $f$ is bounded on $G \cap D(c,s)$.\\ \\

To prove Claim 1, first notice that by Proposition \ref{prop3.10}, it suffices to show that $G$ is a free domain. In the light of equation \eqref{3.40} and the fact that $S$ is finite it will follow that $G$ is a free domain  if we can show that
for each $w \in \c$ and each $C>0$,
\[
\set{x \in \m^1}{ w \not\in \sigma(x) \text{ and } \norm{(w - x)^{-1}} < C} \text{ is a free domain}.
\]
But by Theorem 10.1 in \cite{amfree} for each fixed $w\in \c$, $g_w(x) = (w-x)^{-1}$ is a free holomorphic function on the free domain $\set{x\in \m^1}{w \not\in \sigma(x)}$. As Proposition \ref{prop2.30} guarantees that $g_w$ is freely continuous, it follows that
\[
\set{x \in \m^1}{ w \not\in \sigma(x) \text{ and } \norm{(w - x)^{-1}} < C} = g_w^{-1} (\set{y \in \m^1}{\norm{y}<C})
\]
is a free domain.

To prove Claim 2, let $z \in [c,s]$ and let $x \in G$. As $z \in [c,s]$, inequality \eqref{3.30} guarantees that we may choose  $w \in S$ such that
\be\label{3.50}
|w-z| < \frac{1}{2C}.
\ee
As $x \in G$, equation \eqref{3.40} guarantees that $w \not\in \sigma(x)$ and
\be\label{3.60}
\norm{(w-x)^{-1}} < C.
\ee
Now,
\begin{align*}
(z-x)^{-1}&=\big(w-x)-(w-z)\big)^{-1}\\
&=(w-x)^{-1}\big(1-\frac{w-z}{w-x}\big)^{-1}.
\end{align*}
But as inequalities \eqref{3.50} and \eqref{3.60} imply that
\[
\left\| \frac{w-z}{w-x}\right\| < \frac{1}{2C}\ C = \frac{1}{2},
\]
we can use the inequality $\norm{(1-y)^{-1}} \le (1-\norm{y})^{-1}$ (valid for $\norm{y} <1$) to prove that
\[
\left\|\left(1-\frac{w-z}{w-x}\right)^{-1}\right\| \le \frac{1}{1-\norm{\frac{w-z}{w-x}}}<2.
\]
Hence
\[
\norm{(z-x)^{-1}}\le\norm{(w-x)^{-1}}
\left\|  \left(1-\frac{w-z}{w-x}\right)^{-1}\right\| <2C.
\]

To summarize, we have shown that if $z \in [c,s]$ and $x \in G$, then $\norm{(z-x)^{-1}}\le 2C$. Use the definition  \eqref{3.10} to estimate $f^\wedge(x)$ to deduce that, if $x \in G \cap D(c,s)$, then
\[
\norm{f^\wedge(x)} \le 2C \max_{z\in[c,s]}|f(z)| \text{ length}([c,s]).
\]
This completes the proof of both Claim 2 and the theorem.
\end{proof}

\begin{defin}\label{def3.10}
For any free domain in $U$ in $\m^1$ we define $\uone$ to be $ U \cap \m_1^1$. If $f$ is a free holomorphic function on $U$ we define a function $\fone$ on $\uone$ by
\[
\fone(z) = f([z]), \qquad z \in \uone.
\]
\end{defin}
In the sequel we make no distinction between $z \in \c$ and $[z] \in \m_1^1$, and in particular, view $\uone$ both as a subset of $\c$ and as a subset of $\m^1$.
\begin{prop}\label{prop3.30}
If $U \subseteq \m^1$ is a free domain and $M \in U$, then $\sigma(M) \subseteq \uone$.
\end{prop}
\begin{proof}
Suppose that $M \in U$ and $z \in \sigma(M)$. We wish to show that $z \in \uone$ or equivalently, that $z \in U$.

Choose an $I\times J$ matrix of polynomials $\delta$ such that $M \in B_\delta \subseteq U$ and choose a unit vector $v \in \c^n$ such that $Mv = zv$. We view $\delta(M)$ as a linear transformation from $\c^J \otimes \c^n$ to $\c^I \otimes \c^n$. As $M \in B_\delta$, there exists $r<1$ such that $\norm{\delta(M)} \le r$.

With the setup of the previous paragraph, if $c \in \c^J$,
\begin{align*}
\norm{\delta(z)c} &=\norm{\delta(z)c}\ \norm{v}\\
&=\norm{(\delta (z)c)\otimes v}\\
&=\norm{\delta (M)(c\otimes v)}\\
&\le r \norm{c\otimes v}\\
&= r \norm{c}\norm{v}\\
&= r \norm{c}.
\end{align*}
This proves that $\norm{\delta(z)} \le r <1$. Hence $z \in B_\delta \subseteq U$, as was to be proved.
\end{proof}
\begin{cor}\label{level1}
A nonempty freely open set in $\m^1$ meets $\m^1_1$.
\end{cor}
\begin{prop}\label{prop3.28}
If $U \subseteq \m^1$ is a free domain and $f$ is a free holomorphic function on $U$, then $\fone$ is holomorphic on $\uone$.
\end{prop}
\begin{proof}
Fix $z_0 \in \uone$. As $f$ is free holomorphic, by Theorem \ref{thm2.10} $f$ can be locally uniformly approximated by polynomials. Choose $\delta$ such that $z_0 \in B_\delta \subseteq U$ and  $f$ can be uniformly approximated by polynomials on $B_\delta$. By continuity, there exists $\epsilon >0$ such that $\set{z \in \c}{|z-z_0|<\epsilon} \subseteq B_\delta$. It follows that $\fone$ can be uniformly approximated on a neighborhood in $\c$ by polynomials. Hence $\fone$ is holomorphic on a neighborhood of $z_0$. As $z_0 \in \uone$ was chosen arbitrarily, it follows that $\fone$ is holomorphic on $\uone$.
\end{proof}

\begin{prop}\label{prop3.40}
If $U \subseteq \m^1$ is a free domain and $f$ is a free holomorphic function defined on $U$, then
\[
f(x) = \fone^\wedge(x)
\]
for all $x \in U$.   Moreover $f(x)$ belongs to the algebra generated by $x$ and the identity matrix of appropriate type. 
\end{prop}
\begin{proof}
First observe that the assertion of the proposition makes sense. If $x \in U$, then Proposition \ref{prop3.30} implies that $\sigma(x) \subseteq \uone$. Also, if $f$ is freely holomorphic on $U$, then Proposition \ref{prop3.28} implies that $\fone$ is holomorphic on $\uone$. Thus, $\fone^\wedge(x)$ is well defined by equation \eqref{3.10} for all $x \in U$.

If $p$ is a polynomial, then $p_1^\wedge = p_1$ and $p_1 = p$. Hence, the result holds in the special case where $f=p$. The general case then follows by approximation.
\end{proof}

\section{The Riemann surface for $\sqrt{x}$}\label{SquareRoot}
In this section we shall define $\sqrt{x}$ as a holomorphic function on a one-dimensional free nc-manifold.
By gluing together locally defined branches of the inverse of the free holomorphic function $f(x) = x^2$ we construct a free nc-manifold in much the same way that elementary textbooks construct the Riemann surface for $\sqrt{z}$ by piecing together locally-defined function elements.  
We obtain a locally finitely-sheeted one-dimensional free nc-manifold which has properties analogous to the Riemann surface for $\sqrt{z}$.  

The zero matrix in $\m_2$ has infinitely many square roots, but only one of them lies in the algebra generated by
the zero and identity matrices.  By Proposition \ref{prop3.40}, for any free holomorphic function $f$ and any $x$ for which $f(x)$ is defined, $f(x)$ lies in the algebra $\mathrm{alg} (x)$ generated by $x$ and the identity.  We shall therefore use the symbol $\sqrt{x}$   in the following way.
\begin{defin} \label{defsqrt}
The {\em free square root} $\sqrt{x}$ is the multivalued function on $\m^1$ given by
\[
\sqrt{x}=\{y\in \mathrm{alg} (x) | y^2=x\} \subset \m^1.
\]
\end{defin}
\index{free square root}
 Thus, if $x=0_n$, the $n\times n$ zero matrix then $\sqrt{x}=\{0_n\}$ and if $x=I_n$, the $n\times n$ identity matrix, then $\sqrt{x}=\{I_n,-I_n\}$, whereas if $x=\bbm 0&1\\0&0\ebm$ then $\sqrt{x}=\emptyset$.   We leave the proof of the following proposition to the reader.
\begin{prop}\label{sqrtEmpty}
The free square root of a matrix $x\in\m^1$ is empty if and only if the Jordan canonical form of $x$ contains a
nilpotent Jordan cell of type $k \times k$ for some $k\geq 2$.
\end{prop}
The set
\[
\Xi\df \{x\in\m^1: \sqrt{x} \neq \emptyset\}
\]
\index{$\Xi$}
is not a freely open set:  it contains $0_{2\times 2}$, but any basic free neighborhood of $\Xi$ contains a matrix $\bbm 0&t\\0&0 \ebm$ for some $t\neq 0$.   For the purpose of constructing a Riemann surface we consider the restriction of the square root function to   $\Xi^o$, the interior of $\Xi$ in the free topology.  We shall show (Proposition \ref{Xinterior})
that $\Xi^o$ is the set of nonsingular matrices, and   that the union of the sets $\{x\}\times \sqrt{x}$, as $x$ ranges over all nonsingular matrices, can be given the structure of a free nc-manifold.

Let
\beq\label{definv}
\inv = \set{x \in \m^1}{x \text{ is nonsingular}}
\eeq
\index{$\inv$}
and
\be\label{defQ}
 \q=\set{x\in \m^1}{\sigma(x) \cap \sigma(-x) = \emptyset}.
\ee
\index{$\q$}
\begin{prop}\label{prop4.10}
The sets $\inv$ and $\q$ are open and connected 
in the free topology.
\end{prop}
\begin{proof}
Fix $M \in \inv$. If $q$ is the characteristic polynomial of $M$, then $q(M) = 0$ and $q(0) \neq 0$. If we set $\delta = q(0)^{-1}q$, then clearly, as $\norm{\delta(M)} = 0 <1$, $M \in B_\delta$. Also, if $x \in B_\delta$, then as $\norm{\delta(x)} < 1$,
\[
\delta(\sigma(x)) = \sigma(\delta(x)) \subseteq \d.
\]
Since $\delta(0)=1 \not\in \delta(\sigma(x))$, it follows that $0 \not\in \sigma(x)$, that is $x \in \inv$. Summarizing, given $M \in \inv$, we have constructed a free polynomial $\delta$ such that $M \in B_\delta \subseteq \inv$. Hence $\inv$ is open in the free topology. 

$\inv$ is clearly closed under direct sums, and so it is freely connected, by Proposition \ref{frconn}.

Now fix $M \in \q$. Let $\sigma(M) = \{c_1,\ldots,c_k\}$ and choose $r_1,\ldots,r_k \in {\rplus}$ such that $\Delta(c,r) \cap \Delta(-c,r) = \emptyset$. Let $q$ be the characteristic polynomial of $M$ and define $\delta = \rho q$ where $\rho \in \rplus$ is chosen so large that
\[
\set{z\in \c}{|\delta(z)|<1} \subseteq \Delta(c,r).
\]
With these choices, $M \in B_\delta$ and if $x \in B_\delta$ then $\sigma(x) \subseteq \Delta(c,r)$. Hence, as  then $\sigma(-x) \subseteq \Delta(-c,r)$ as well, $\sigma(x) \cap \sigma(-x) = \emptyset$, that is, $x \in \q$. Summarizing, given $M \in \q$, we have constructed a free polynomial $\delta$ such that $M \in B_\delta \subseteq \q$. Hence $\q$ is open in the free topology.  

Consider any points $x, y\in \q$.  Although $\q$ is not closed under $\oplus$, since $\si(x)$ can meet $\si(-y)$, it is nearly so.  Choose $t\in (0,1]$ such that $\si(tx)$ is disjoint from $\si(-y)$; then $tx\oplus y \in\q$.  Remark \ref{slightlymore} shows that there is a freely continuous path in $\q$ from $tx$ to $y$, while there is an obvious freely continuous path in $\q$ from $x$ to $tx$.
\end{proof}

\begin{prop}\label{Xinterior}
The interior $\Xi^o$ of the set $\Xi$ in the free topology is $\inv$.
\end{prop}
\begin{proof}
By Proposition \ref{prop4.10}, $\inv$ is freely open, and it is clearly contained in $\Xi$.  Suppose $\inv$ is a proper subset of $\Xi^o$; then there is a singular matrix $M\in \Xi$ and a basic free neighborhood $B_\delta$ of $M$ contained in $\Xi$.  Since $B_\delta$ is invariant under unitary conjugations, we may assume that $M$ is upper triangular, and since $M$ is singular, we may take $M$ to have zero as its $(1,1)$ entry.   Let $N=M\oplus M$; then $N\in B_\delta$, and for a suitable permutation matrix $P$, the matrix $P^*NP$ in $B_\delta$ is upper triangular and has $0_{2\times 2}$ as a block in the $(1,1)$ position.  For some complex $\zeta\neq 0$, $B_\delta$ contains the upper triangular matrix $T$ differing from $P^*NP$ only in that its $(1,2)$ entry is $\zeta$.  If $e_1, e_2, \dots$ denotes the standard basis of $\c^n$, for the appropriate $n$, then a Jordan chain for $T$ corresponding to the eigenvalue $0$ is $e_2, \zeta e_1$.  Hence the Jordan form of $T$ has a nilpotent Jordan cell of type at least $2\times 2$, and therefore, by Proposition \ref{sqrtEmpty}, $\sqrt{T}$ is empty, contradicting the fact that $B_\delta \subset \Xi$.  Hence $\Xi^o= \inv$.
\end{proof}

We shall construct the Riemann surface for $\sqrt{x}$ by piecing together function elements over $\inv$.
\begin{defin}\label{def4.10}
By a \emph{free function element over $\inv$}
\index{free!function element}
 is meant a pair $(f,U)$ where $U$ is a free domain in $\inv$ and $f$ is a free holomorphic function on $U$. We say a function element $(f,U)$ is a \emph{branch of $\sqrt{x}$} if $f(x)^2 = x$ for all $x\in U$.
\end{defin}
\begin{lem}\label{lem4.10}
Let $(f,U)$ and $(g,V)$ be free function elements over $\inv$, both assumed to be branches of $\sqrt{x}$. If $M \in U \cap V$ and $f(M) = g(M)$, then there exists a free domain $W$ such that $M \in W \subseteq U \cap V$ and $f|W = g|W$.
\end{lem}
\begin{proof}
Let $(f,U)$ and $(g,V)$ be branches of $\sqrt{x}$ and assume that $M \in U \cap V$ and $f(M) = g(M)$. By Proposition \ref{prop3.40} there exist holomorphic functions $\fone$ on $\uone$ and $\gone$ on $\vone$ such that $f=\fone^\wedge$ on $U$ and $g=\gone^\wedge$ on $V$.

Now, since $M \in U\cap V$, by Proposition \ref{3.30}
\be\label{}
\sigma(M) \subseteq \uone \cap \vone.\notag
\ee
Furthermore, since $f$ and $g$ are branches of $\sqrt{x}$ on $U$ and $V$ respectively,
\be\label{4.10}
\fone,\gone \text{ are branches of } \sqrt{z} \text{ on } \uone \cap \vone.
\ee
Let $\sigma(M) = \{c_1,\ldots,c_k\}$ and choose $r \in {\rplus}^k$ so that $\Delta(c,r) \subseteq \uone \cap \vone$. It follows from equation  \eqref{4.10} that
\be\label{4.20}
\fone,\gone \text{ are branches of } \sqrt{z} \text{ on } \Delta(c,r).
\ee

But since $f(M) = g(M)$, $\fone = \gone$ on $\sigma(M)$. Since each component of $\Delta(c,r)$ meets $\sigma(M)$, it follows from  equation \eqref{4.20} that $\fone = \gone$ on $\Delta(c,r)$. Therefore, by Proposition \ref{prop3.40},
\[
f(x) = \fone^\wedge (x) = \gone^\wedge(x) = g(x)
\]
for all $x \in D(c,r)$. Since Proposition \ref{prop3.10} guarantees that $D(c,r)$ is a free domain, the lemma follows by choice of $W=U\cap V \cap D(c,r)$.
\end{proof}

 The following definition expresses the Riemann surface for $\sqrt{x}$ as a union of graphs of function elements. This approach follows Chapter 8 of \cite{ahl} quite closely. An alternative approach, based on cross-sections of a sheaf of germs of free holomorphic functions over $\inv$, is also possible. However, in the simple special case we are considering, this latter approach would amount to little more than a change in notation.
\begin{defin}\label{def4.20}
If $(f,U)$ is a free function element, let
\[
\gra (f,U) = \set{(x,f(x))}{x \in U}.
 \]
\index{$\gra(f,U)$}
Let $S$ denote the collection of all branches $(f,U)$ of $\sqrt{x}$ where $U$ is a basic free open set $B_\delta$ in $\inv$ for some matricial free polynomial $\delta$. Define $\calr$ by
\[
\calr = \bigcup_{(f,U) \in S} \gra{(f,U)}
\]
\index{$\calr$}
and define $\calb$ by
\[
\calb = \set{\gra (f,U)}{(f,U) \in S}.
\]
\end{defin}
\index{$\calb$}
\begin{lem}\label{lem4.20}
There exists a unique topology $\t$ on $\calr$ such that $\calb$ is a basis for $\t$.
\end{lem}
\begin{proof}
The statement is that, if $(f,U), (g,V) \in S$ and
\[
(M,N) \in \gra(f,U)\cap \gra(g,V),
\]
then there exists $(h,W) \in S$ such that
\[
(M,N) \in \gra(h,W) \subseteq \gra(f,U)\cap \gra(g,V).
\]
This assertion follows immediately from Lemma \ref{lem4.10}.
\end{proof}

\begin{defin}\label{def4.30}
Let $\calr$ be equipped with the topology $\t$ of Lemma {\rm \ref{lem4.20}}.
For $(f,U) \in S$, define
\[
\alpha_{(f,U)}:\gra (f,U) \to U
\]
by the formula
\[
\alpha_{(f,U)}(x,f(x)) = x, \qquad x \in U.
\]
Let
\[
\a = \set{\alpha_{(f,U)}}{(f,U) \in S}.
\]
\end{defin}
\begin{thm}\label{thm4.10}
$(\calr,\t,\a)$ is a free manifold.
\end{thm}
\begin{proof}
The theorem follows from the following four facts each of which is a simple consequence of the previous constructions:
\begin{enumerate}
\item For each $(f,U) \in S$, $\gra (f,U) \in \t$,
\item For each $(f,U) \in S$, $U$ is a free domain,
\item For each $(f,U) \in S$, $\alpha_{f,U}:\gra (f,U) \to U$ is a homeomorphism from $\gra (f,U)$ equipped with the $\t$ topology to  $U$ equipped with the free topology.
    \item If $(f,U),(g,V) \in S$ and $\gra (f,U) \cap \gra (g,V)  \neq  \emptyset$, then $\alpha_{(g,V)} \circ \alpha_{(f,U)}^{-1}$ is a free holomorphic function.
\end{enumerate}
1. holds as $\gra (f,U) \in \calb$ if $(f,U) \in S$. \\ \\According to Definition \ref{def4.20}, if $(f,U) \in S$ then $U$ is a free domain. \\So 2. holds. \\ \\To see 3., assume that $V$ is a free domain in $U$. Then
\[
\alpha_{(f,U)}^{-1}(V) = \gra (f|V,V) \in \calb \subseteq \t.
\]
Conversely, if $\gra(V,g)$ is a basic $\t$-open set in $\gra(f,U)$, then, as $V\subseteq U$ and $g=f|V$,
\[
\alpha_{(f,U)}(\gra(g,V)) = V \text{ is a free domain.}
\]
4. follows from the fact that if $x \in \alpha(\gra (f,U) \cap \gra (g,V) )$\\ then $\alpha_{(g,V)} \circ \alpha_{(f,U)}^{-1}(x) = x$.
\end{proof}

As in the classical case, the point of the Riemann surface for a multivalued function $f$ on a domain $D$ is that $f$ can be regarded as a single-valued holomorphic function on the Riemann surface, which is a holomorphic manifold lying over $D$.  The following statement makes this notion precise in the context of the matricial square root.  Consider a point $w\in\calr$: then $w\in\gra(f,U)$ for some function element $(f,U)$, which is to say that $w=(x,f(x))$ for some $x\in U$.  We shall say that $w$ {\em lies over} $x$.
\begin{thm}\label{singlevalued}
There is a holomorphic function $F$ on the free holomorphic manifold $(\calr,\t,\a)$ such that, if $w\in\calr$ lies over $x\in\inv$, then $F(w)$ is a square root of $x$.
\end{thm}
\begin{proof}
Define $F$ to be the restriction to $\calr$ of the co-ordinate projection $(M^1,M^2) \mapsto M^2$.  It is easy to check that $F$ is holomorphic on $\calr$.  If $w$ lies over $x\in\inv$ then $w=(x,f(x))$ for some branch $f$ of the square root on a neighborhood of $x$; then $F(w)=f(x)$, which is a square root of $x$.
\end{proof}

It is interesting to observe that, in contrast to the commutative case, where the Riemann surface for $\sqrt z$ lies over $\c \setminus \{0\}$ as a 2-sheeted surface, in the noncommutative case there is no bound on the number of sheets in $\calr$ that lie over a given point $M \in \inv$.
\begin{prop}\label{prop4.20}
Let $M \in \inv$ and let $\sigma(M)$ have $k$ elements. There exist exactly $2^k$ points $N \in \m^1$ such that $(M,N) \in \calr$.
\end{prop}
\begin{proof}
The proof uses the simple observation from linear algebra that a matrix $M$ has exactly $2^k$ square roots that lie in $\mathrm{alg}(M)$, the algebra generated by $M$. Fix $M \in \inv$ and assume that $|\sigma(M)|=k$.

If $(M,N) \in \calr$, then there exists $(f,U) \in \s$ such that $(M,N) \in \gra(f,U)$. Then
\[
N = f(M) = \fone^\wedge(M) \in \mathrm{alg}(M).
\]
Hence, there exist at most $2^k$ matrices $N$ such that $(M,N) \in \calr$.

To see that there exist at least $2^k$ matrices $N$ such that $(M,N) \in \calr$, let $\sigma = \{c_1,\ldots,c_k\}$ and choose $r \in {\rplus}^k$ so that $\Delta(c,r)$ has $k$ components and $0 \not\in \Delta(c,r)$. Each of the $2^k$ distinct choices of square roots $b_1,\ldots,b_k$ for the points $c_1,\ldots,c_k$ gives rise to a distinct holomorphic branch $h_b$ of $\sqrt z$ on $\Delta(c,r)$ satisfying $h_b(c_i) = b_i$, $i=1,\ldots,k$. In turn each of these distinct holomorphic branches of  $h_b$ gives rise to a distinct function element $(h_b^\wedge,\Delta(c,r)) \in S$. As $\sigma(h_b^\wedge(M)) =  \{b_1,\ldots,b_k\}$, this proves that there are at least $2^k$ matrices $N$ such that $(M,N) \in \calr$.
\end{proof}
Equally interesting is to observe that despite the phenomenon described in the preceding proposition, $\calr$ is isomorphic to a free domain in $\m^1$.
\begin{prop}\label{prop4.30}
The map $\sigma:\q \to \calr$ defined by the formula 
\[
\sigma(y) = (y^2,y) \qquad \mbox{ for } y \in \q
\]
 is a free biholomorphism from $\q$ onto $\calr$.
\end{prop}
\begin{proof}
Clearly, $\sigma$ is injective and onto. If $(f,U) \in S$, then
\[
(\alpha_{(f,U)} \circ \sigma)\ (y) = \alpha_{(f,U)}(y^2,y) = y^2
\]
is a free holomorphic function defined on $\set{y}{(y^2,y) \in \gra(f,U)}$ and
\[
(\sigma^{-1} \circ \alpha_{(f,U)}^{-1})\ (x) = \sigma^{-1}(x,f(x)) = f(x)
\]
is a free holomorphic function defined on $\set{x}{x \in U}$. Therefore $\sigma$ is a free biholomorphic mapping.
\end{proof}

\section{The Zariski-free topology}\label{zariski}
 We come now to a modification of the free topology on $\m^d$ that will be needed for the construction of the topological nc-manifold $\g$ with properties described in Theorems \ref{thm5.10} and \ref{thment}.

\subsection{Thin sets}
\label{subsec7.1}
Recall that a set $T$ in a domain $U \subseteq \c^d$ is said to be \emph{thin}
\index{set!thin}
 if at each point $z \in U$ there exists an open neighborhood $V$ of $z$ in $U$ and a nonconstant holomorphic function $f$ on $V$ such that $f=0$ on $V \cap T$. Simple facts are that thin sets are nowhere dense, closures of thin sets are thin sets and finite unions of thin sets are thin sets. 

 A more subtle property of thin sets will be fundamental in later sections: if $T$ is a thin set in $U$ then every  holomorphic function $f$  on $U\setminus T$ that is locally bounded on $U$  has a unique holomorphic extension to all of $U$ \cite[Theorem I.3.4]{ran86}.  Here the boundedness hypothesis  is defined as follows.

\begin{defin}\label{def2.110}
 Let $U$ be a domain in $\c^d$ and let $T\subset\c^d$.  A function $f$ on $U\setminus T$  is said to be {\em locally bounded on $U$}
\index{function!locally bounded}
 if, for every point $z\in U$ there is a neighborhood $V$ of $z$ in $U$ such that $f$ is bounded on $V\setminus T$.
\end{defin}

It is false that every holomorphic locally bounded function on $U\setminus T$, where $T$ is thin and $U$ is a domain in $\c^d$, has a holomorphic extension to $U$.  An easy counterexample is $U=\c, T=\{0\}, \, f(z)=z^{-1}$.
It is essential that $f$ be locally bounded {\em on $U$} for a holomorphic extension to exist.

\subsection{The Zariski-free topology}

Let $D$ be a free domain.   By a \emph{free variety in $D$} we mean a set $V \subseteq D$ that has the form
\be \label{freevariety}
V = \set{x \in D}{ f(x) =0 \mbox{ for all }f\in S}
\ee
for some set $S$ of  freely holomorphic functions on $D$.  
 Note that the zero symbol in equation \eqref{freevariety} stands simultaneously for square zero matrices of all orders.  To make matters precise we define
\[
\mathbf {0}\df \{0_{1}, 0_{2}, 0_{3},\ldots,\}
\]
where $0_n$ denotes the zero matrix in $\mn$.    Thus the definition \eqref{freevariety} can equally be written
\[
V = \set{x \in D}{ f(x) \in \mathbf{0} \mbox{ for all }f\in S}.
\]
We shall be loose about distinguishing $\mathbf {0}$ and $0$.
%
%
%
%

\black
The following statement illustrates just one of the many surprises that result from the free topology not being Hausdorff.
Free varieties in a free domain $D$ are not necessarily relatively closed in $D$ in the free topology.
\begin{prop}\label{prop2.50}
For any freely holomorphic function $p$ on a free domain $D\subset \m^1$, the free closure of the free variety  
\[
T=\{x\in D: p(x)=0\}
\]
 in $D$  is the set
\[
\set{x \in D}{p(x) \text { is singular}}.
\]
\end{prop}

\begin{proof} 
Let $x\in D\cap\mn$ be such that $p(x)$ is singular.  We claim that $x$ is in the closure $T^-$ of $T$.  

Since $p(x)$ is singular, there is an eigenvalue $\la$ of $x$ such that $p(\la)=0$.  Let $u\in\c^n$ be a corresponding eigenvector.
Consider any  basic free neighborhood $B_\delta$ of $x$ in $D$, where $\delta$ is an $m\times m$ matrix of polynomials in one variable; then $\delta(x)\in \m_{mn}$ and $\norm{\delta(x)}_{\m_{mn}}< 1$.
  For any choice of $\zeta=(\zeta_1,\ldots,\zeta_m) \in \C^m$,
\[
\delta_{ij}(x) \zeta_j u = \delta_{ij}(\la) \zeta_j u \quad \mbox{ for } i,j=1,\ldots,m.
\]
Let $\zeta\otimes u$ denote the $nm\times 1$ matrix $[\zeta_1u^T  \ldots \zeta_m u^T]^T$, where the superscript $T$ denotes transposition.  The last equation shows that 
\begin{align*}
\delta(x) (\zeta\otimes u) &= \bbm \de_{11}(x) & \dots & \de_{1m}(x) \\ \cdot & \dots &\cdot \\ \de_{m1}(x)& \dots & \de_{mm}(x) \ebm
	\bbm \zeta_1 u \\  \vdots \\ \zeta_m u \ebm \\
	&=  \bbm \de_{11}(\la) & \dots & \de_{1m}(\la) \\ \cdot & \dots &\cdot \\ \de_{m1}(\la)& \dots & \de_{mm}(\la) \ebm
	\bbm \zeta_1 u \\ \vdots \\ \zeta_m u \ebm \\
	&= (\delta(\la)\zeta)\otimes u.
\end{align*}
Take norms of both sides in $\c^{mn}$:
\begin{align}\label{nmboth}
\|\de(x)(\zeta\otimes u)\|_{\c^{mn}} &= \|(\delta(\la)\zeta)\otimes u\|_{\c^{mn}} \notag \\
	&= \|\de(\la)\zeta\|_{\c^m} \; \|u\|_{\c^n}.
\end{align}
Choose $\zeta$ to be a unit maximizing vector for $\delta(\la)$, so that
\begin{align}\label{sothat}
\|\zeta\|_{\c^m}=1 \quad\mbox{ and }\quad \|\de(\la)\zeta\|_{\c^m} = \|\de(\la)\|_{\m_m}.
\end{align}
Combining equations \eqref{nmboth} and \eqref{sothat}, we find that
\[
\|\de(x)(\zeta\otimes u)\|_{\c^{mn}} =  \|\de(\la)\|_{\m_m} \; \|u\|_{\c^n}.
\]
Since  $\|\de(x)\| < 1$,
\begin{align*}
\norm{\delta(\la)}_{\m_m} \; \norm{u}_{\c^n} & \leq \norm{\delta(x)}_{\m_{mn}}  \; \|\zeta\otimes u\|_{\c^{mn}} \\
	& <  \|\zeta\|_{\c^m} \, \|u\|_{\c^n} \\
	&= \|u\|_{\c^n},
\end{align*}
and hence 
\[
\norm{\delta(\la)}_{\m_m} < 1.
\]
That is, $\la \in G_\de$.  Since $p(\la)=0, \; \la\in T$.  We have shown that every basic free neighborhood of $x$ in $D$ meets $T$, and so $x$ is in the free closure of $T$.

Conversely, suppose that $x\in D\cap\mn$  and $p(x)$ is nonsingular.  Then there is a basic free neighborhood $B_\delta$ of $x$ in $\m^1$ that is disjoint from $T$.
Indeed, we may choose $\delta$ to be the $1\times 1$ polynomial $c(g\circ p)$ where $g$ is the characteristic polynomial of $p(x)$ and $c > 1/|\det p(x)|$.    By the Cayley-Hamilton theorem, $\delta(x)=cg(p(x))=0$, so that $x\in B_\delta$.  For any $m\in\N$ and $M\in T\cap \m_m$,
\[
\de(M)=cg(p(M))=cg(0_m)= cg(0_1) 1_m= c(\det p(x)) 1_m,
\]
and therefore
\[
\norm{\delta(M)}=\norm{c(\det p(x)) 1_m} = c|\det p(x)|> 1,
\]
that is, $M\notin G_\de$.  Thus $D\cap \gdel$ is a free  neighborhood of $x$ in $D$ disjoint from $T$.  Thus $x\notin T^-$.
\end{proof}
\begin{cor}\label{nonclosed}
If $p$ is a nonconstant free polynomial in one variable then the free variety $p^{-1}(\mathbf{0})$ is not freely closed in $\m^1$.
\end{cor}
For we can easily write down a $2\times 2$ diagonal matrix $M$ such that $p(M)$ is singular but not $0_2$.
\begin{cor}\label{closT}
For any free polynomial $p$ in $d$ variables, the free closure of the free variety  
\[
T=\{x\in\m^d: p(x)=0\}
\]
 in $\m^d$  is contained in the set
\[
\set{x \in \m^d}{p(x) \text { is singular}}.
\]
\end{cor}
\begin{proof}
Let $\mathrm{Sing}$ denote the set of singular matrices in $\m^1$.  By Proposition \ref{prop2.50} and the free continuity of $p, \; p^{-1}(\mathrm{Sing})$ is a freely closed set in $\m^d$.  It clearly contains the set
 $T=p^{-1}(\mathbf{0})$.  Hence it contains the free closure of $T$.
\end{proof}

Many of the important results about varieties in commutative analysis depend critically on the fact that varieties are relatively closed.  Accordingly the following modification of the free topology is natural and fruitful in the nc context, as we shall see in Subsection \ref{thedefofg}.
\begin{defin}\label{def2.80}
A \emph{basic Zariski-free set} in $\m^d$ is a set $B$ that has the form
\[
B = D \setminus T
\]
where $D$ is a free domain in $\m^d$ and
\[
T = \bigcup_{i=1}^k V_i
\]
is a finite union of  free varieties in $D$.
\end{defin}
\index{topology!Zariski-free}


In the definition, the set $T$ is taken to be a finite union of free varieties. This is because, in contrast to the commutative case, a finite union of free varieties is not in general a free variety.

\begin{prop}\label{zfbase}
The collection of basic Zariski-free  sets in $\m^d$ constitutes a base for a topology on $\m^d$.
\end{prop}
\begin{proof}
Clearly the collection covers $\m^d$.  Consider any pair $D_j\setminus T_j, \ j=1,2$, of basic \zf sets in $\m^d$.  Observe that, for any freely open subset $D$ of $D_j$, the intersection $D\cap T_j$ is a finite union of free varieties in $D$, since the restriction to $D$ of a freely holomorphic function on $D_j$ is freely holomorphic on $D$. 

The set-theoretic identity 
\be\label{setthid}
(D_1\setminus T_1) \cap (D_2\setminus T_2) = (D_1 \cap D_2) \setminus \left((T_1 \cup T_2) \cap D_1\cap D_2\right),
\ee
thus implies that the collection of basic Zariski-free  sets in $\m^d$ is closed with respect to finite intersections and hence forms a base for a topology.
\end{proof}

\begin{defin}\label{def2.90}
The \emph{Zariski-free topology on $\m^d$}
\index{topology!Zariski-free}
 is the topology that has as a base the collection of basic Zariski-free sets. A set that is open in the Zariski-free topology is a \emph{Zariski-free domain}.
\end{defin}
Note that there is also a smaller base for the Zariski-free topology, consisting of the sets $B_\delta\setminus T$ where $\delta$ ranges over  all matricial free polynomials and $T$ ranges over all finite unions of  free varieties in $B_\delta$.  This base has the additional feature that it consists of nc-sets.   Consequently the topology is an nc topology.  It is clearly coarser than the finitely open topology.  We deduce, in the terminology of Definition \ref{def2.40}:
\begin{prop}\label{Zfadmiss}
The Zariski-free topology is an admissible topology on $\m^d$.   
\end{prop}
An empty union is conventionally taken to be the empty set.  Thus every freely open set in $\m^d$ is also \zfly open. 
\begin{prop}\label{zffiner}
The \zf topology is finer than the free topology on $\m^d$.
\end{prop}

Any free variety, like a free open set, is preserved by unitary conjugations.  It follows that Zariski-free open sets are also preserved by unitary conjugations.

We now explore some important relationships between free and Zariski-free domains and holomorphic functions.
 \begin{prop}\label{prop2.60}
If $D_1$ and $D_2$ are free domains and $f:D_1 \to D_2$ is a free holomorphic mapping 
then $f$ is continuous in the Zariski-free topology.
 \end{prop}
 \begin{proof}
Let $D_1\subset \m^{d_1}, \ D_2\subset \m^{d_2}$.
 Fix a basic Zariski-free set $B = D \setminus T \subseteq \m^{d_2}$, where $D \subseteq D_2$ is a free domain and $T$ is a finite union of   free varieties in $D$.
   Suppose that, as in Definition \ref{def2.80}, $T=V_1\cup\dots\cup V_k$ where $V_i = \set{x \in D}{ g(x) =0 \mbox{ for all } g\in S_i }$ and
$S_i$ is a set of free holomorphic functions on $D$.  Then
 \[
 f^{-1}(V_i) = \set{x \in f^{-1}(D)}{ g(f(x)) =0 \mbox{ for all } g\in S_i }.
 \]
Therefore, if we let $S_i^\sim = \set{g\circ f}{g \in S_i}$, then
\[
f^{-1}(V_i) = \set{x \in f^{-1}(D)}{ h(x) =0\mbox{ for all } h\in S_i^\sim}
\]
is a free variety in $f^{-1}(D)$. Consequently
\[
f^{-1}(T) = \bigcup_{i=1}^k \set{x \in f^{-1}(D)}{ h(x) =0 \mbox{ for all }h\in S_i^\sim}
\]
is a finite union of free varieties in $f^{-1}(D)$. 
The set 
 \[
 f^{-1}(B) = f^{-1}(D) \setminus f^{-1}(T)
 \]
is therefore Zariski-freely open.
 \end{proof}

\begin{lem}
\label{lemlib1}
Let $\phi$ be a Zariski-freely holomorphic function on $D \setminus T$, where $D$ is a freely open
nc set,
 $T$ is a finite union of free varieties in $D$, and $D \setminus T$ is non-empty.
Then $\phi$ extends to a unique freely holomorphic nc function
$\hp$ on $D$.
\end{lem}
\begin{proof}
Let $z \in D \cap T$. Choose $w \in D \setminus T$.
 Then $z\oplus w \in D \setminus T$, and so $\ph(z\oplus w)$ is defined.
For each component of the $d$-tuple $z\oplus w$
\[
\bbm 1&0\ebm z^j\oplus w^j = z^j \bbm 1&0\ebm.
\]
Hence, by the intertwining property \eqref{2.607}
of nc functions,
\[
\bbm 1&0\ebm \ph(z\oplus w) = \ph(z) \bbm 1&0\ebm.
\]
Similarly, for each $j$,
\[
z^j\oplus w^j \bbm 0\\ 1\ebm = \bbm 0\\ 1\ebm w^j
\]
and therefore
\[
\ph(z\oplus w) \bbm 0\\ 1\ebm = \bbm 0\\ 1\ebm \ph(w).
\]
Hence
\be\label{gota}
\ph(z\oplus w) = \bbm a&0\\0&\ph(w) \ebm
\ee
for some matrix $a$ of the same type as the $z^j$.
Likewise, if $w_1$ and $w_2$ are distinct elements of $D \setminus T$,
then since
\[
 \begin{bmatrix}
 1&0 \\
 0 & 0
 \end{bmatrix}
\
 \begin{bmatrix}
 z&0 \\
 0 & w_1
 \end{bmatrix}
 \ = \
 \begin{bmatrix}
 z&0 \\
 0 & w_2
 \end{bmatrix}
 \
 \begin{bmatrix}
 1&0 \\
 0 & 0
 \end{bmatrix},
 \]
the intertwining property \eqref{2.607} tells us that
the value of $a$ in equation \eqref{gota}
 does not depend on the choice of $w$ in $D \setminus T$.
Hence we may define a graded function $\hp$ on $D\cap T$ by the relation
\be\label{gothp}
\ph(z\oplus w) = \bbm \hp(z)&0\\0&\ph(w) \ebm
\ee
for all $z\in D\cap T$ and $w\in D\setminus T$.

Extend $\hp$ to $D$ by defining $\hp$ to agree with $\ph$ on $D\setminus T$.
We claim that $\hp$ is a freely holomorphic extension of $\phi$ from $D \setminus T$ to $D$.
That $\hp (z_1 \oplus z_2) = \hp(z_1) \oplus \hp(z_2)$ is immediate if both $z_1 $ and $z_2$ are
in $D \setminus T$, since $\phi$ is an nc function by Proposition \ref{holnc}.  It is also immediate if just one of $z_1,z_2$ is in $D \setminus T$, by choice of the other as $w$.
Assume therefore that both $z_1$ and $z_2$ are in $D\cap T$, and choose any $w\in D\setminus T$.
The relation
\be\label{hpzz}
\ph\left(\bbm z_1\oplus z_2 & 0\\0& w\ebm\right) = \bbm \hp(z_1\oplus z_2)&0\\0&\ph(w) \ebm,
\ee
is true by the  definition of $\hp$ if  $z_1\oplus z_2 \in T$, and by the fact that $\hp$ extends $\ph$ together with the nc property of $\ph$ if $z_1\oplus z_2 \notin T$.

Since $z_2\oplus w \in D\setminus T$,
\begin{align}\label{hpz1}
\ph\left(\bbm z_1\oplus z_2 &0\\0&w\ebm\right)&=\ph\left(\bbm z_1&0\\0&z_2\oplus w\ebm\right)\notag \\
	&=\bbm\hp(z_1) &0\\ 0&\ph(z_2\oplus w)\ebm \notag \\
	&= \bbm \hp(z_1) & 0&0\\0&\hp(z_2)&0\\0&0&\ph(w)\ebm.
\end{align}
Comparison of equations \eqref{hpzz} and \eqref{hpz1} reveals that 
\[
\hp(z_1\oplus z_2)= \hp(z_1) \oplus \hp(z_2).
\]

To see that $\hp$ preserves similarity, consider $x\in D$ and $s\in\inv$ such that $s^{-1}xs\in D$.
If $x\notin T$ then also $s^{-1}xs\in D\setminus T$, and so 
\[
 \hp(s^{-1} z s) = \ph (s^{-1} z s) = s^{-1}\ph( z) s = s^{-1}\hp( z) s.
\]
Now suppose that $x\in D\cap T$; then $s^{-1}xs\in D\cap T$.   Choose any $w\in D\setminus T$. We have
\begin{align*}
\hp(s^{-1}xs)\oplus \ph(w)&= \ph(s^{-1}xs\oplus w) \qquad\qquad \mbox{ by definition of} \hp\\
	&=\ph\left((s\oplus 1)^{-1}(x\oplus w)(s\oplus 1)\right) \\
	&=(s\oplus 1)^{-1}\ph(x\oplus w)(s\oplus 1) \qquad \mbox{ since $\ph$ is nc}\\
	&=(s\oplus 1)^{-1}(\hp(x)\oplus \ph(w))(s\oplus 1)\\
	&= s^{-1}\hp(x) s \oplus \ph(w).
\end{align*}
Hence $\hp(s^{-1}xs)=s^{-1}\hp(x) s$.

It remains to show that $\hp$ is freely locally bounded, hence freely holomorphic.
Let $z$ be any point in $D$, and choose $w \in D \setminus T$.
There is a Zariski-free basic neighborhood $\gdel \setminus S$ containing $z \oplus w$
on which $\phi$ is bounded, by $M$ say.
For every $x \in \gdel$,
\[
 \phi( x \oplus z \oplus w) \ = \
 \hp(x) \oplus \hp(z) \oplus \phi (w)
 \]
 is bounded by $M$, so $\hp$ is bounded on the free neighborhood $\gdel$ of $z$.

 To see that the extension is unique, it is sufficient to prove that if $\psi$ is a free nc holomorphic function
 on $D$ that vanishes on $D \setminus T$, then it is identically zero.
 Suppose not. Then there is some point $z \in T \cap D$ such that $\psi(z) \neq 0$.
There is some point $w \in D \setminus T$. Then $z \oplus w$ 
 is in $D \setminus T$, but $\psi(z \oplus w) = \psi(z) \oplus \psi(w) \neq 0$,
 a contradiction.
 \end{proof}

\begin{lem}
\label{lemlib2}
Let $\phi : \gdel \setminus T \to \ggam \setminus S$ be a  Zariski-freely holomorphic nc map.
Then $\phi$ extends to a unique freely holomorphic map $\hp : \gdel \to \ggam$.
Moreover, if $\phi$ is a Zariski-free homeomorphism, then $\hp$ is a free homeomorphism.
\end{lem}
\begin{proof}
Since $B_\ga$ is an nc set, we may apply Lemma~\ref{lemlib1} to each component of $\ph$ in turn to obtain a freely holomorphic nc map
 $\hp$ on $B_\ga$.
We must show that the range of $\hp$ is contained in $\ggam$.
This is automatic for points in $\gdel \setminus T$, so consider a point $x \in T \cap \gdel$. Suppose $\| \gamma ( \hp (x) ) \| \geq 1$.
Then there are unit vectors $u,v$ such that
\be
\label{eq1}
| \langle  \gamma ( \hp (x) ) u, v \rangle | \ \geq \ 1 .
\ee
Let $w \in \gdel \setminus T$.
Then
\[
| \langle  \gamma ( \hp (x) ) u, v \rangle |
\ = \
| \langle  \gamma ( \phi (x \oplus w) ) u \oplus 0, v \oplus 0 \rangle | \ < \ 1,
\]
a contradiction.

The uniqueness of $\hp$ follows by Lemma~\ref{lemlib1}.

Suppose $\psi :  \ggam \setminus S \to \gdel \setminus T$ is the inverse of $\phi$.
Let $w \in \gdel \setminus T$.
Then for all $z \in \gdel$, we have $z \oplus w \in B_\de\setminus T$, and
\begin{eqnarray*}
\begin{bmatrix}
 z&0 \\
 0 & w
 \end{bmatrix} & \ = \ &
\psi \circ \phi \left(
\begin{bmatrix}
 z&0 \\
 0 & w
 \end{bmatrix}
\right) \\
& \ = \ &
\psi
 \left(
\begin{bmatrix}
 \hp(z)&0 \\
 0 & \phi(w)
 \end{bmatrix}
\right) \\
& \ = \ &
\begin{bmatrix}
\hat{\psi} \circ  \hp(z)&0 \\
 0 & \psi \circ \phi(w)
 \end{bmatrix}.
\end{eqnarray*}
Thus $\hat{\psi}$ is the inverse of $\hp$.
\end{proof}
\begin{lem}
\label{lemlib3}
Let $\phi$ be a Zariski-freely holomorphic function on a Zariski-free open set $U \subseteq \m^d$.
Suppose that $U$ can be written as
\[
U \ = \
\bigcup_{\alpha \in A} \ (\bda \setminus T_\alpha ),
\]
where $T_\alpha$ is a finite union of free varieties in $\bda$ such that

{\rm (i)} whenever $\bda \cap \bdb$ is non-empty, so is
$(\bda \setminus T_\alpha) \cap (\bdb \setminus T_\beta)$;

{\rm (ii)} the function $\phi$ is bounded and nc on each set $\bda \setminus T_\alpha$ for $\al\in A$.

Then $\phi$ extends to a unique freely holomorphic function $\hp$ on 
\[
\hat{U} \ := \ 
\bigcup_{\alpha \in A} \ \bda.
\]
\end{lem}
\begin{proof}
By Lemma~\ref{lemlib1}, $\phi$ can be extended to a unique free function on each $\bda$.
To see that these extensions coincide on each $\bda \cap \bdb$, we observe that, by condition (i), if this intersection
is non-empty, then there is a point $w$ in  $(\bda \setminus T_\alpha) \cap (\bdb \setminus T_\beta)$.
We can use this point $w$ in equation \eqref{gothp} to construct the extension of $\ph$ to both $\bda$ and $\bdb$, and so the
extension will agree on the intersection.
\end{proof}

\black

The next result shows how close Zariski-free holomorphy is to free holomorphy.
\begin{prop}\label{veryclose}
A graded function on a free domain in $\m^d$ is freely holomorphic if and only if it is  \zfly holomorphic.
\end{prop}
\begin{proof}
Let $U$ be a free domain in $\m^d$ and let $\ph: U\to \m^1$ be a graded function.  Suppose that $\ph$ is freely holomorphic -- that is, it is freely locally nc and freely locally bounded on $U$.  Since the \zf topology is finer than the free topology, it is immediate that $\ph$ is \zfly locally nc and \zfly locally bounded.  Thus $\ph$ is \zfly holomorphic.

Conversely, suppose that $\ph$ is  \zfly holomorphic on $U$.  
Let $x \in U$. There is a Zariski-free open set $D \setminus T$ containing $x$ on which $\ph$ is nc and bounded.
Therefore there exists $\gdel$ such that $\gdel \setminus T$ contains $x$ and $\phi$ is nc and bounded on 
$\gdel \setminus T$. 
By Lemma~\ref{lemlib1}, there is a bounded nc extension $\hat{\ph}$ of $\ph |_{\gdel \setminus T}$ to $\gdel$.
If $y \in \gdel \cap T$, then
\begin{eqnarray*}
\hat{\ph}(x \oplus y) &\ = \ & \ph( x \oplus y) \\
\shortparallel\quad &&\quad  \shortparallel \\
\hat{\ph}(x) \oplus \hat{\ph}(y)  &\ =\  &\ph(x) \oplus \ph (y) .
\end{eqnarray*}
So $\hat{\ph}$ agrees with $\ph$ on all of $\gdel$, including $T$. Therefore $\ph$ is freely locally bounded.
\end{proof}

\section{An nc-manifold for symmetrization } \label{UniversalDomain}

In this section we construct a  two-dimensional topological nc-manifold $\g$ such that the algebra of holomorphic functions on $\g$ is canonically isomorphic (in a sense to be made precise in Theorem \ref{thment} below) to the algebra of symmetric free holomorphic functions in $\m^2$.  

In the commutative case, if $\Omega\subset \C^2$ is a symmetric domain (that is, if $(z,w)\in\Omega$ implies that $(w,z) \in\Omega$) then the domain $\tilde\Omega \df \{(z+w,zw):(z,w)\in\Omega\}$ has the property that the symmetric holomorphic functions $f$ on $\Omega$ are in bijective correspondence with the holomorphic functions $F$ on $\tilde\Omega$ via the relation
\[
f(z,w) =F(z+w,zw) \quad \mbox{ for all }(z,w)\in \Omega.
\]
We shall construct a topological nc-manifold $\g$ with properties analogous to those of $\tilde\Omega$ for the  symmetric free domain $\Omega= \m^2$.

In \cite{ay14} the algebra of symmetric free holomorphic functions on the biball $B^2$, the noncommutative analog of the bidisc, is  analysed  by means of operator-theoretic methods.
The authors constructed a set $\Omega \subset \m^\infty$ and a map  $\pi:B^2 \to\Omega$ such that every symmetric free holomorphic function $\ph$ on $B^2$ can be expressed as $\Phi\circ \pi$ for some holomorphic nc-function $\Phi$ on
$\Omega$.  Here $\Omega$ is an infinite-dimensional set, but it can nevertheless be described in terms
 of only three noncommuting variables, provided that inverses and square roots are allowed.  In this paper we adopt a more topological approach and obtain a {\em $2$-dimensional} topological nc-manifold with properties analogous to $\Omega$, but for symmetric functions on an arbitrary symmetric free domain.  The theme that three variables suffice is reflected in the fact that the manifold we obtain in this section is presented as a subset of $\m^3$.  However the topology of the manifold structure of $\g$ is not that induced by any natural topology of $\m^3$.

\subsection{A geometric lemma}

In this and the next two subsections we describe some simple combinatorial geometry and properties of free square roots, which play an essential part in the construction of $\g$. 

Throughout Section \ref{UniversalDomain} the symbol $\Delta$ will be reserved for sets in the plane that have the form
\[
\Delta = \bigcup_{i=1}^k (c_i + r\d).
\]
Alternatively, if $\gamma$ denotes the finite set $\{c_1,c_2,\ldots c_k\}$,
\[
\Delta = \gamma + r\d.
\]
\index{$\Delta$}
We refer to such sets as \emph{simple sets with radius $r$}. 
\index{set!simple}
We define the \emph{separation} of a simple set $\Delta = \gamma + r\d$, denoted  by $\sep \Delta$, by
\index{$\sep \Delta$}
\[
\sep \Delta = \min \{| c-d|: c,d\in \gamma, c\neq d\}.
\]
\index{separation}
\begin{defin}\label{def5.10}
A simple set $\Delta$ with radius $r$ is \emph{$t$-isolated} if $r < t \sep \Delta$. If $\Delta_1$ and $\Delta_2$ are simple sets, then $\Delta_2$ is {\em subordinate} to $\Delta_1$ if each disc in $\Delta_2$ meets at most one disc in $\Delta_1$.
\end{defin}

\begin{lem}\label{lem5.10}
Let $\Delta_1$ and $\Delta_2$ be simple sets with radii $r_1$ and $r_2$ respectively. If $\Delta_1$ and $\Delta_2$ are $\tfrac 14$-isolated, then either $\Delta_1$ is subordinate to $\Delta_2$ or $\Delta_2$ is subordinate to $\Delta_1$.
\end{lem}
\index{subordinate}
\begin{proof}
Suppose not, so that neither $\Delta_j$ is subordinate to the other.  Since $\Delta_1$ is not subordinate to $\Delta_2$ there is a disc $c+r_1\d$ in $\Delta_1$ that meets two discs $d_j+r_2\d$, $ j=1,2$ in $\Delta_2$, with $d_1\neq d_2$.  Pick points $\zeta_1, \zeta_2$ such that
\[
\zeta_j \in (c+r_1\d) \cap (d_j+r_2\d), \qquad j=1,2.
\]
Then
\[
|\zeta_j-c| < r_1\ \text{ and }\  |\zeta_j-d_j| < r_2.
\]
Since $\Delta_1, \Delta_2$ are $\tfrac 14$-isolated,
\[
r_1 < \tfrac 14 \sep \Delta_1\ \text{ and }\  r_2 < \tfrac 14 \sep \Delta_2,
\]
and so
\[
|\zeta_j-c| < \tfrac 14 \sep \Delta_1\ \text{ and }\  |\zeta_j-d_j| < \tfrac 14 \sep \Delta_2.
\]
Therefore,
\[
|c-d_j| < \tfrac 14\sep\Delta_1 +  \tfrac 14\sep\Delta_2, \qquad j=1,2.
\]
In consequence,
\begin{align*}
\sep\Delta_2 &\leq |d_1-d_2| \\
	&\leq |c-d_1| +|c-d_2|\\
	&<\tfrac 12\sep\Delta_1 +\tfrac 12\sep\Delta_2.
\end{align*}
Hence,
\be\label{5.10}
\sep \Delta_2 < \sep \Delta_1.
\ee
Since also $\Delta_2$ is not subordinate to $\Delta_1$, by repeating the above argument with $\Delta_1$ and $\Delta_2$ swapped, we deduce that $\sep\Delta_1 < \sep\Delta_2$, contradicting inequality \eqref{5.10}.
\end{proof}

\subsection{Holomorphic square roots on the sets $\Delta_\ga$} 
Let $\Gamma$ denote the set of finite subsets of $\c \setminus \{0\}$. 
\index{$\Ga$}
For each $\gamma \in \Gamma$ we fix throughout the remainder of the section a simple set
\[
\Delta_\gamma = \bigcup_{c \in \gamma} (c + r_\gamma \d)
\]
\index{$\Delta_\ga$}
where $r_\gamma$ is chosen so that
\[
r_\gamma < \min\ \{\min_{c \in \gamma} |c|,\tfrac 14 \min_{\substack{c,d \in \gamma\\ c \neq d}} |c-d|\ \}.
\]
This choice of $r_\gamma$ guarantees that, for all $\gamma \in \Gamma$,
\be\label{5.20}
0 \notin \Delta_\gamma
\ee
and
\be\label{5.30}
 \Delta_\gamma \text{ is } \tfrac 14 \text{-isolated}.
\ee
Notice that the statement \eqref{5.30} implies that $\Delta_\gamma$ is a finite union of open discs whose closures are pairwise disjoint.

For $\gamma \in \Gamma$ let $|\gamma|$ denote the cardinality of $\gamma$. Define $2^{|\gamma|}$ functions on $\Delta_\gamma$ in the following way. For each $\tau \in \{-1,1\}^\gamma$, define $\iota_{\gamma\tau}:\Delta_\gamma \to \{-1,1\}$ by the formula
\[
\iota_{\gamma\tau}(z)=\tau(c) \text{ if }  z \in c + r_\gamma \d.
\]
\index{$\iota\subgt$}
The functions $\iota_{\gamma \tau}$ are holomorphic on $\Delta_\gamma$, indeed, they precisely consist of the $2^{|\gamma|}$ holomorphic square roots of the constant function $1$ on $\Delta_\gamma$.

For each $\gamma \in \Gamma$, property \eqref{5.20} and the remark following statement \eqref{5.30} imply that there exists a branch of $\sqrt z$ defined on $\Delta_\gamma$. 

\begin{convention}\label{defsgam}
For every $\ga\in\Ga$, fix a holomorphic branch $s_\ga$ of the square root function on $\Delta_\ga$.
\end{convention}
\black
Once this $s_\gamma$ is chosen, the other branches of $\sqrt z$ on $\Delta_\ga$ can be described in terms of $s_\gamma$ with the aid of the functions $\iota\subgt$; we define $s\subgt$ to be the holomorphic square root function $s_\ga\iota\subgt$ on $\Delta_\ga$.

\begin{lem}\label{lem5.20}
Assume $\gamma \in \Gamma$. A function $f$ on $\Delta_\gamma$ is a holomorphic function  satisfying $f(z)^2 = z$ for all $z \in \Delta_\gamma$ if and only if there exists $\tau \in \{-1,1\}^\gamma$ such that $f = s\subgt$.
\end{lem}

\subsection{Free square roots on the sets $ D_\gamma$} \label{subsec8.3}
For each $\gamma \in \Gamma$ define $D_\gamma \subseteq \m^1$ by
\[
D_\gamma = \set{x \in \m^1}{\sigma(x) \subseteq \Delta_\gamma}.
\]
\index{$D_\gamma$}
\begin{lem}\label{lem5.30}
\[
\bigcup_{\gamma \in \Gamma} D_\gamma = \inv
\]
\end{lem}
\begin{proof}
If $M\in D_\gamma$, then by definition, $\sigma(M) \subseteq \Delta_\gamma$. Hence property \eqref{5.20} implies that $0 \not\in \sigma(M)$, that is, $M \in \inv$.  Conversely, if $M \in \inv$, then $0 \not\in \sigma(M)$. This implies that $\sigma(M) \in \Gamma$. Hence $M \in D_{\sigma(M)} \subseteq \cup_{\gamma \in \Gamma} D_\gamma$.
\end{proof}
With the help of the functions $s_\gamma$ and $\iota\subgt$ we may define free holomorphic functions on  $D_\gamma$ by means of the Riesz functional calculus. Since $x \in D_\gamma$ implies that $\sigma(x) \subseteq \Delta_\gamma$, we may use the formula \eqref{3.10} to define,  for every $\ga\in\Ga$, every $\tau\in\{-1,1\}^\ga$ and every $x\in D_\ga$,
\begin{align}
S_\gamma (x) &= s_\gamma^\wedge(x),   \label{defSga}\\
I\subgt (x) &= \iota\subgt^\wedge(x), \label{defIgt}\\
S\subgt(x) &=S_\gamma (x)I\subgt (x). \label{defSgt}
\end{align}
\index{$I\subgt$}
\index{$S\subgt$}
Since $s_\ga(z)^2=z$ and $\iota\subgt (z)^2=1$ for $z\in \Delta_\ga$,
\[
S\subgt(x)^2=x \quad \mbox{ and }\quad I\subgt(x)^2=1 \quad \mbox{ for all } \ga\in\Ga \mbox{ and } x \in D_\ga.
\]
Notice that, by Proposition \ref{prop3.3}, $S\subgt$ and $I\subgt$ are free holomorphic functions on the free domain $D_\ga$.
Moreover, {\em every} square root of $x\in D_\ga$ in the algebra generated by $x$ has the form $S\subgt(x)$ for some $\tau\in\{-1,1\}^\ga$.
\subsection{The Zariski-free domain $\ugam$}
\index{$\ugam$}
For each $\gamma \in \Gamma$ we shall define a Zariski-free open set in $\m^2$ having a certain genericity property.  Let
\be\label{5.32}
\wgam = \set{(u,x) \in \m^2}{x \in D_\gamma}.
\ee
\index{$\wgam$}
As $D_\gamma$ is a free domain, so is $\wgam$. Furthermore, for fixed $\gamma \in \Gamma$ and $\tau \in \{-1,1\}^\gamma$, as $I\subgt$ is a free holomorphic function,
\[
f\subgt (u,x) = u \ I\subgt (x) - I \subgt (x)\ u, \qquad (u,x) \in \wgam,
\]
defines a free holomorphic function on $\wgam$. It follows that
\[
V\subgt = \set{(u,x) \in \wgam}{u \ I\subgt (x) = I \subgt (x)\ u}
\]
is a free variety in $\wgam$.
\index{$V\subgt$}

In the constant cases, where either $\tau = 1$ (that is, $\tau(c) = 1$ for all $c \in \gamma$) or $\tau = -1$ (that is, $\tau(c) = -1$ for all $c \in \gamma$), we have either $I\subgt = 1$ or $I\subgt =-1$. Thus, in these two cases where $\tau$ is constant, $V\subgt$ is all of $\wgam$. In the sequel, we express the condition that $\tau$ is not constant by writing $\tau \neq \pm 1$.

Since, for each $\tau \in \{-1,1\}^\gamma$, $V\subgt$ is a  free variety in $\wgam$,
\be\label{5.38}
\ugam = \wgam \setminus \bigcup_{\substack{\tau \in \{-1,1\}^\gamma\\ \tau \neq \pm 1}} V\subgt
\ee
defines a set that is open in the Zariski-free topology of $\m^2$, though not, generally, in the free topology, since the varieties $V\subgt$ are typically not freely closed. We shall construct the Zariski-free manifold $\g$ by gluing together sheets in $\m^3$ that lie over the domains $\ugam\subset \m^2$.  We can express the definition of $\ugam$ as follows:
\begin{align} \label{altdefugam}
\ugam &= \{(u,x)\in\m^2:  \ x\in D_\ga \mbox{ and $u$ does not commute with } I\subgt(x) \notag \\
	& \hspace*{1.5cm} \mbox{ for any nonconstant } \tau\in \{-1,1\}^\ga \}.
\end{align}
The following statement is easy to see.
\begin{prop}\label{propUgam}
For any $\ga \in \Gamma$, if $(u_1,x_1) \in \ugam$ and $(u_2,x_2) \in \wgam$ then $(u_1,x_1)\oplus (u_2,x_2) \in \ugam$.
\end{prop}
\subsection{The definition of $\g$}\label{thedefofg}
To motivate the ensuing definition of the set $\g$ we recount some ideas from \cite{ay14}.
It has long been known \cite{wo36} that the algebra $A$ of symmetric complex polynomials in two noncommuting variables $w^1$ and $w^2$ is not finitely generated.  However, it was found in \cite{ay14} that there is a close substitute for a finite basis.  If we define
\[
u=\half(w^1+w^2), \qquad v=\half(w^1-w^2)
\]
then an algebraic basis of $A$ is
\be\label{infbasis}
u, \, v^2, \, vuv, \, vu^2v,\dots,vu^jv,\dots
\ee
Although this is an infinite basis of $A$, the relation
\[
vu^jv=(vuv) (v^2)^{-1} (vuv) (v^2)^{-1}\dots (v^2)^{-1}(vuv)
\]
shows that every symmetric free polynomial in $w^1,w^2$ can be written as a {\em rational}
expression in terms of the first three terms of the basis \eqref{infbasis}.  Accordingly, it appeared that these three basic
polynomials might have the potential to play the role that the elementary symmetric functions 
$w^1+w^2$ and $w^1w^2$ play in the scalar theory, though some extra complications result from the fact that rational expressions are required to represent polynomials.

Thus,  the underlying idea is to study the
image of $\m^2$ under the map $(u,v^2,vuv)$.  Here, we realise this approach.  We write
$x$ for the variable $v^2$ (so that $v$ is a square root of $x$), and then invoke the structure theory for free holomorphic square roots developed in Section \ref{SquareRoot}.  
Our strategy for the construction of a topological nc-manifold with the desired `universal symmetrization' property will be to
 apply the maps $(u,x,\sqrt{x} u\sqrt{x})$ to $\m^2$.  The expectation is that the many branches of the square root lead to co-ordinate patches on the image set that can be pieced together to yield the required topological nc-manifold.  Our first attempt
to execute this strategy and thereby construct a {\em free} nc-manifold failed because of singular behavior on certain subvarieties.  The notion of the Zariski-free topology
enabled us to circumvent this difficulty -- see Remark \ref{whyzf} below.

Let us first establish that the polynomials \eqref{infbasis} do indeed constitute an algebraic basis for the algebra of free polynomials.
We do not know whether this is a new observation.  A closely related result of an analytic flavor is \cite[Theorem 5.1]{ay14}.
\begin{thm}\label{thebasis}
Let $x,y$ be non-commuting indeterminates and let
\[
u=\half(x+y), \qquad v=\half(x-y).
\]
For any positive intger $d$, every free polynomial of total degree $d$ in $x,y$ can be written as a polynomial in the $d$ elements $u, v^2, vuv, \dots, vu^{d-2}v$.
\end{thm}
\begin{proof}
It suffices to prove the result for {\em homogeneous} free polynomials.  For $d\geq 1$ let $P_d$ be the complex vector space of homogeneous polynomials of total degree $d$ in $x,y$ and let $\sym_d$ be its subspace of symmetric polynomials.  Clearly $P_d$ has dimension $2^d$, and therefore $\sym_d$ has dimension $2^{d-1}$.

Let $Q_d$ be the space of polynomials in   $u, v^2, vuv, \dots, vu^{d-2}v$ that are homogeneous of degree $d$ in $u,v$, and hence also in $x,y$.
Then $Q_1=\c u$ and $Q_d\subseteq \sym_d$.  We claim that $\dim Q_d= 2^{d-1}=\dim \sym_d$, from which it will follow that $Q_d=\sym_d$, as required.

The claim is true when $d=1$ since $\dim Q_1=1$.  Let $d>1$ and suppose the claim holds for $d-1$.  We have
\[
Q_d=uQ_{d-1}\oplus v^2 Q_{d-2}\oplus vuv Q_{d-3} \oplus  \dots \oplus vu^{d-3}v Q_1\oplus vu^{d-2}v\c.
\]
By the inductive hypothesis,
\begin{align*}
\dim Q_d &= \dim Q_{d-1} + \dim Q_{d-2} + \dots +\dim Q_1 + 1\\
	&=2^{d-2}+2^{d-3} + \dots + 1+1\\
	&=2^{d-1}.
\end{align*}
Hence $Q_d=\sym_d$ and the theorem follows.
\end{proof}

For each $\gamma \in \Gamma$ and each $\tau \in \{-1,1\}^\gamma$ we define a mapping $\Phi\subgt:\ugam \to \m^3$ by the formula
\be\label{5.40}
\Phi\subgt (u,x) = (u, x,S\subgt(x)\ u\ S\subgt (x)), \qquad (u,x) \in \ugam,
\ee
\index{$\Phi\subgt$}
where $S\subgt$ is the free holomorphic square root on $D_\ga$ defined in equation \eqref{defSgt}.

Let
\be\label{5.41}
\g\subgt = \ran \Phi\subgt.
\ee
\index{$\g\subgt$}
Trivially, $\Phi\subgt$ is a bijection from $\ugam$ to $\g\subgt$.
Define $\g \subseteq \m^3$ by
\be\label{5.42}
\g = \bigcup_{\gamma \in \Gamma}\ \  \bigcup_{\tau \in \{-1,1\}^\gamma} \g\subgt.
\ee
\index{$\g$}
\subsection{A Zariski-free atlas for $\g$ }
The set $\g$ defined in the previous subsection is just that, a set. It carries neither a topology nor an atlas of charts that would endow it with a manifold structure. In this section we shall topologize $\g$ and equip it with a Zariski-free atlas.

\begin{defin}\label{def5.20}
Let
\[
\calb = \set{\Phi\subgt(U)}{\gamma \in \Gamma, \tau\in \{-1,1\}^\gamma, U \text{ is Zariski-freely open in } \ugam}.
\]
\end{defin}
\index{$\calb$}
Proposition \ref{lem5.50} will imply that $\calb$ is a base for a topology on $\g$.  The proof  hinges on the following technical statement.
\begin{lem}\label{whyweneedzariski}
Let $(u,x)\in \ugamsub{1} \cap \ugamsub{2}$ for some $\ga_1,\ga_2\in\Ga$.  If, for some $\tau_1\in \{-1,1\}^{\ga_1}$ and $\tau_2\in \{-1,1\}^{\ga_2}$,
\be\label{SuS}
S\subgtsub{1}(x)uS\subgtsub{1}(x) = S\subgtsub{2}(x)uS\subgtsub{2}(x)
\ee
then
\be\label{eitheror}
 \text{ either } \quad 
 S\subgtsub{1}(x) = S\subgtsub{2}(x)  \quad \text{ or } \quad S\subgtsub{1}(x) = - S\subgtsub{2}(x). 
\ee
\end{lem}

\begin{proof}
Since $(u,x)\in \ugamsub{1} \cap \ugamsub{2}$ we have
\[
x \in D_{\gamma_1} \cap D_{\gamma_2},
\]
or equivalently,
\be\label{5.50}
\sigma(x) \subseteq \Delta_{\gamma_1} \cap \Delta_{\gamma_2}.
\ee
In the light of property \eqref{5.30} and Lemma \ref{lem5.10} we may assume that $\Delta_{\gamma_2}$ is subordinate to $\Delta_{\gamma_1}$, that is, that each component of $\Delta_{\gamma_2}$ meets at most one component of $\Delta_{\gamma_1}$. This implies that  each component of $ \Delta_{\gamma_2}$ contains at most one component of $\Delta_{\gamma_1} \cap \Delta_{\gamma_2}$.  
As both of the functions $s_{\gamma_1\tau_1}$ and $s_{\gamma_2\tau_2}$ are branches of $\sqrt z$ on $\Delta_{\gamma_1} \cap \Delta_{\gamma_2}$,  they agree up to a factor $\pm1$ on each component of $\Delta_{\gamma_1} \cap \Delta_{\gamma_2}$.  Hence there exists $\tau \in \{-1,1\}^{\gamma_2}$ such that
\be\label{5.70}
s\subgtsub{1}(z) = \iota_{\gamma_2\tau}(z)s\subgtsub{2}(z), \qquad z \in \Delta_{\gamma_1} \cap \Delta_{\gamma_2}.
\ee

 Equations \eqref{5.50} and \eqref{SuS} imply that
 \be\label{SIS}
 S\subgtsub{1}(x) = I_{\gamma_2\tau}(x)S\subgtsub{2}(x).
 \ee
Substitution  of this formula for $ S\subgtsub{1}(x)$ into equation \eqref{SuS} yields
\[
u = I_{\gamma_2\tau}(x)uI_{\gamma_2\tau}(x).
\]
But $I_{\gamma_2\tau}(x)$ is an involution. Therefore
\[
f_{\gamma_2\tau}(u,x)=uI_{\gamma_2\tau}(x)-I_{\gamma_2\tau}(x)u =0,
\]
that is, $(u,x) \in V_{\gamma_2\tau}$. As $(u,x) \in \mathcal{U}_{\gamma_2}$, this implies that $\tau$ is constant on $\gamma_2$ and either $\tau = 1$ or $\tau = -1$. We deduce that either $\iota_{\gamma_2\tau}\equiv 1$ or $\iota_{\gamma_2\tau}\equiv -1$ on $\Delta_{\gamma_1} \cap \Delta_{\gamma_2}$.  In conjunction with equation \eqref{SIS}, this implies that equation \eqref{eitheror} holds.
\end{proof}

\begin{remark} \label{whyzf} \rm 
The foregoing technical lemma explains the introduction of the Zariski-free topology.  It is only because of the exclusion of the varieties $V\subgt$, where a non-generic commutation relation holds, in the definition of the sets $\ugam$ that the collection $\calb$ of Definition \ref{def5.20} constitutes a base for a topology on $\g$.
\end{remark}
\black
\begin{prop}\label{lem5.50}
If $\Omega_1, \Omega_2 \in \calb$ and $\omega \in \Omega_1 \cap \Omega_2$, then there exists $\Omega_3 \in \calb$ such that $\omega \in \Omega_3 \subseteq \Omega_1 \cap \Omega_2$.
\end{prop}
\begin{proof}
Fix $\Omega_1 = \Phi\subgtsub{1}(U_1)$, $\Omega_2 = \Phi\subgtsub{2}(U_2)$ and assume that $(u,x,y) \in \Omega_1 \cap \Omega_2$. As
\[
(u,x) \in U_1 \cap U_2 \subseteq \ugamsub{1} \cap \ugamsub{2},
\]
and
\be\label{5.60}
S\subgtsub{1}(x)uS\subgtsub{1}(x) = y=S\subgtsub{2}(x)uS\subgtsub{2}(x)
\ee
we may apply Lemma \ref{whyweneedzariski} to deduce that
\be\label{plusorminus}
\mbox{  either } S\subgtsub{1}(x) = S\subgtsub{2}(x) \mbox{ or } S\subgtsub{1}(x) = - S\subgtsub{2}(x).
\ee
Let
\[
\Omega_3 = \Phi_{\gamma_2\tau_2} (U_1 \cap U_2).
\]
As $(u,x) \in U_1\cap U_2$, equation \eqref{5.60} implies that
\[
(u,x,y) \in \Omega_3.
\]
Clearly
\[
\Omega_3 =\Phi_{\gamma_2\tau_2} (U_1 \cap U_2) \subseteq \Phi_{\gamma_2\tau_2} ( U_2) = \Omega_2.
\]
To see that $\Omega_3 \subseteq \Omega_1$, fix $(u,x) \in U_1 \cap U_2$. Then, as $\sigma(x) \subseteq \Delta_1 \cap \Delta_2$, equation \eqref{plusorminus} implies that
\begin{align*}
\Phi_{\gamma_2\tau_2}(u,x) &=(u,x,S\subgtsub{2}(x)\  u \ S\subgtsub{2}(x))\\
&=(u,x,S\subgtsub{1}(x)\  u\  S\subgtsub{1}(x))\\
&=\Phi_{\gamma_1\tau_1}(u,x) \\
&\ \in \Omega_1.
\end{align*}
\end{proof}

 Lemma \ref{lem5.50} implies that $\calb$ is a base for a unique topology $\t$ on $\g$.
The topology $\t$ is defined so that the maps
\[
\Phi\subgt ^{-1}:\g\subgt \to \ugam
\]
are homeomorphisms when $\g\subgt$ carries the $\t$ topology and $\ugam$ carries the Zariski-free topology. As $\ugam$ is a Zariski-free domain, it follows that
\be\label{5.90}
\a = \set{\Phi\subgt ^{-1}}{\gamma \in \Gamma, \tau \in \{-1,1\}^\gamma}
\ee
is a collection of topological nc-co-ordinate patches on $\g$ (see condition \eqref{2.110}).

In fact $\a$ is a Zariski-free atlas for $(\g,\t)$. Since Definition \eqref{5.42} implies that the sets $\g\subgt$ cover $\g$,  equation \eqref{2.98} holds. To see that the 
transition functions are Zariski-freely holomorphic, observe that
\[
(\Phi\subgtsub{2}^{-1} \circ \Phi\subgtsub{1})(u,x) = (u,x)
\]
for all $(u,x) \in \Phi\subgtsub{1}^{-1}(\g\subgtsub{1} \cap \g\subgtsub{2})$,
and $\Phi\subgtsub{1}^{-1}(\g\subgtsub{1} \cap \g\subgtsub{2})$ is Zariski-freely open
by Lemma~\ref{lem5.50}.

We summarize the above observations.
\begin{thm}\label{thm5.10}
If $\g$ is the set defined by equation \eqref{5.42}, $\t$ is the topology on $\g$ generated by the set $\calb$ in Definition {\rm \ref{def5.20}} and $\a$ is given by equation \eqref{5.90} then $(\g,\t,\a)$ is a Zariski-free manifold.
\end{thm}

Our original hope was to construct a free nc-manifold $\g$ and a surjective free holomorphic map $\pi:\m^2 \to \g$ with the property that, for every symmetric free holomorphic function $\ph$ on $\m^2$, there exists a free holomorphic function $\Phi$ on $\g$ such that $\ph=\Phi\circ\pi$.  Our construction above falls short of this goal on at least two counts.   Firstly, $\g$ is not a free manifold, but a {\em Zariski}-free manifold.  Secondly, if we make the intended definition
\begin{align}\label{defpi}
\pi(w)&=(u,v^2,vuv)   \notag\\
	&= (\tfrac 12(w^1+w^2), \tfrac 14 (w^1-w^2)^2, \tfrac 18 (w^1-w^2)(w^1+w^2)(w^1-w^2))
\end{align}
\index{$\pi$}
then a necessary condition that $\pi(w) \in\g$ is that $v^2 \in\mathcal{I}$, since we require $(u,v^2)$ to belong to some $\ugam$; 
thus $\pi$ does not map $\m^2$ to $\g$.  Nonetheless, there is a
 correspondence between $\ph$ and a suitable holomorphic function $\Phi$ on $\g$, but the correspondence is more subtle, as the next  statement shows.
 
\begin{thm}
 \label{thment}
 There is a canonical bijection between  the classes of 
 
 {\rm (i)}
 symmetric nc functions $f$ that are freely holomorphic
 on $\m^2$, and
  
{\rm  (ii)} holomorphic functions $F_{oo}$ defined on the Zariski-free manifold $\g$ that are \anc\ and 
 have the property that, for every $w \in \m^2$, there is a free neighborhood $U$ of $w$ such that
 $F_{oo}$ is bounded on $\pi(U) \cap \g$.
 \end{thm}

Theorem~\ref{thment} will be proved in the next section, in greater generality -- see Theorem \ref{thm9.2}.

\section{Symmetric free holomorphic functions} 
\label{SymmetricFunctions}

In Section  \ref{UniversalDomain}  we constructed an nc-manifold $\g$ for the representation of freely {\em entire} symmetric functions on $\m^2$.
In this section we shall do likewise for symmetric free holomorphic functions on an arbitrary symmetric free domain.

Recall that a set $\s \subseteq \m^2$ is \emph{symmetric} if
\[
w=(w^1,w^2) \in \s \implies (w^2,w^1) \in \s.
\]
In the sequel $\s$ is a fixed symmetric free domain in $\m^2$.  We shall construct a Zariski-free manifold $\gg_{oo}(\s)$ which is, roughly speaking, the restriction of $\g$ to $\s$. 

All the notations and constructions in Section \ref{UniversalDomain}   are in effect in this section as well. For $w \in \m^2$ we shall frequently employ the change of variables 
\be\label{uandv}
u = \half(w^1 +w^2),\qquad v = \half(w^1-w^2),
\ee
or, equivalently,
\[
w^1 = u+v, \qquad w^2 = u-v.
\]
The operation of tranposition of components will again  be denoted by $(w^1,w^2)^f=(w^2,w^1)$.

To construct $\g(\s)$ we need to define several sets and mappings, as indicated schematically in the following figure.
\begin{center}
\includegraphics {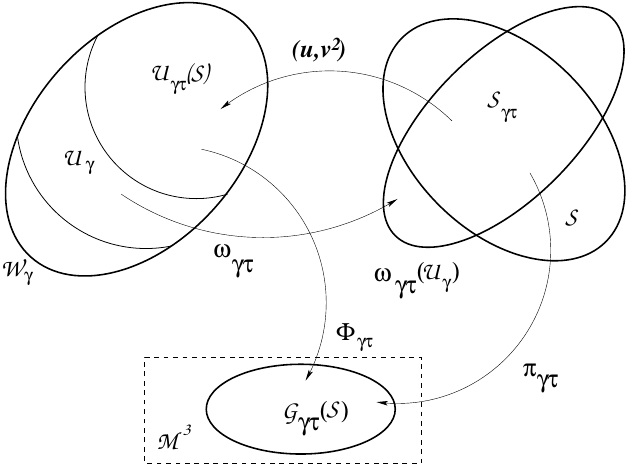}\\
Scheme for the construction of $\g(\s)$
\end{center}

\subsection{Local inverses of $(u,v^2)$} 
The purpose of this and the next two subsections is to define a submanifold of the Zariski-free manifold $\g$ corresponding to a symmetric free domain $\s$ in $\m^2$.  First we introduce maps $\omega\subgt$ which are local inverses of the map $w\to (u,v^2)$, in the notation of equations \eqref{uandv}.

Recall that, for each $\gamma \in \Gamma$, the set $\wgam$ defined in equation \eqref{5.32} is  freely open in $\m^2$. For each $\gamma \in \Gamma$ and each $\tau \in \{-1,1\}^\gamma$ we
define a mapping $\omega\subgt:\wgam \to \m^2$ by the formula
\be\label{6.10}
\omega\subgt (u,x) = \Big(u+S\subgt(x)\ ,\ u-S\subgt(x)\Big),\qquad (u,x) \in \wgam
\ee
where $S\subgt$ is the free holomorphic function on $D_\gamma$ defined by equation \eqref{defSgt}.

Let $w = \omega\subgt(u,x)$, so that
\[
u+S\subgt(x) = w^1 \qquad \text{and} \qquad
u-S\subgt(x) = w^2.
\]
Solution of these equations for $u$ and $x$ gives
\[
u = \tfrac12 (w^1+w^2) \qquad \text{and}\qquad  x=S\subgt (x)^2 = \tfrac14 (w^1-w^2)^2.
\]
Thus $\omega\subgt$ is injective and
\begin{align}\label{6.13}
\omega\subgt^{-1}(w) &= \Big(\tfrac12 (w^1+w^2),\tfrac14 (w^1-w^2)^2 \Big)\\
	&=(u,v^2) \notag
\end{align}
for all $w \in \ran \omega\subgt$.
\begin{prop}\label{prop6.10}
For each $\gamma \in \Gamma$ and $\tau \in \{-1,1\}^\gamma$, $\omega\subgt$ is a free biholomorphic mapping from $\wgam$ onto its range.
\end{prop}
\begin{proof}
 We first  show that $\ran \omega\subgt$ is open in the free topology.  Equation \eqref{6.13} defines a free polynomial map $F$ on all of $\m^2$. Hence, as this map is continuous in the free topology,  $\ran \omega\subgt = F^{-1}(\wgam)$  
is open in the free topology.

 Once it is known that both the domain and range of $\omega\subgt$ are freely open sets, that $\omega\subgt$ is a  free   biholomorphic mapping follows immediately from the formulas \eqref{6.10} and \eqref{6.13}.
\end{proof}
\subsection{The sets $\s\subgt$ and $\calu\subgt(\s)$}
Recall that, for each $\gamma \in \Gamma$, the set  $\ugam$ defined in equation \eqref{5.38} is a Zariski-free open subset of $\wgam$. For each $\tau \in \{-1,1\}^\gamma$, we define $\s\subgt \subseteq \s$ by
\be\label{6.15}
\s\subgt = \omega\subgt (\ugam) \cap \s
\ee
and define $\calu\subgt(\s) \subseteq \ugam$ by 
\begin{align}
\calu\subgt(\s)&= \set{(u,x) \in \ugam}{\omega\subgt(u,x) \in \s} \notag \\
&=\omega\subgt^{-1}(\s\subgt). \label{6.16}
\end{align}
\begin{prop}\label{prop6.20}
$\calu\subgt(\s)$ and $\s\subgt$ are Zariski-free  open sets and
$\omega\subgt$ is a  Zariski-free  biholomorphic mapping from $\calu\subgt(\s)$ onto $\s\subgt$.
\end{prop}
\begin{proof}
By Proposition \ref{prop6.10}, $\omega\subgt:\wgam \to \ran \omega\subgt$ is a free biholomorphic mapping. It follows by Proposition \ref{prop2.60}  that $\omega\subgt$ is a homeomorphism when $\wgam$ and $\ran \omega\subgt$ are equipped with the Zariski-free  topology. Hence equation \eqref{6.15} implies that $\s\subgt$ is Zariski-freely open and equation \eqref{6.16} implies that $\calu\subgt(\s)$ is Zariski-freely open.

That $\omega\subgt$ acting on $\calu\subgt(\s)$ and $\omega\subgt^{-1}$ acting on $\s\subgt$ are Zariski-free  holomorphic mappings (that is, are Zariski-freely locally bounded) follow from the facts that  $\omega\subgt$ acting on $\wgam$ and $\omega\subgt^{-1}$ acting on $\ran \omega\subgt$ are free holomorphic mappings (and hence are  freely locally bounded).
\end{proof}
\subsection{The sets $\g\subgt(\s)$ and $\g(\s)$}
Recall that for each $\gamma \in \Gamma$ and each $\tau \in \{-1,1\}^\gamma$,  $\Phi\subgt$ and $\g\subgt$ are defined by equations \eqref{5.40} and \eqref{5.41}. For each $\gamma$ and $\tau$, we define
\[
\g\subgt (\s) = \Phi\subgt(\calu\subgt(\s)),
\]
\index{$\g\subgt(\s)$}
and then, following the definition \eqref{5.42}, we set
\be\label{6.20}
\gg_{oo}(\s) = \bigcup_{\gamma \in \Gamma}\ \  \bigcup_{\tau \in \{-1,1\}^\gamma} \g\subgt(\s).
\ee
\index{$\gg_{oo}(\s)$}
Recall that the topology $\t$ on $\g$ was chosen so that the maps $\Phi\subgt:\ugam \to \g\subgt$ are homeomorphisms when $\ugam$ carries the Zariski-free  topology. Therefore $\gg_{oo}(\s)$ is an open subset of $\g$ and as such carries the structure of a Zariski-free  manifold. A Zariski-free atlas for $\gg_{oo}(\s)$, $\a(\s)$, can be obtained from the Zariski-free atlas $\a$ for $\g$ defined in equation \eqref{5.90} by simple restriction of the charts in $\a$ to the sets $\g\subgt (\s)$, that is, by the definition
\be\label{6.30}
\a (\s)= \{\Phi\subgt ^{-1} \left| \g\subgt(\s)\right.:\gamma \in \Gamma, \tau \in \{-1,1\}^\gamma\}.
\ee
\index{$\a(\s)$}
We summarize these observations in the following theorem.
\begin{thm}\label{thm6.10}
Let $\s$ be a symmetric free domain in $\m^2$. If $\gg_{oo}(\s)$ is defined by equation \eqref{6.20}, then $\gg_{oo}(\s)$ is an open subset of $\g$. Furthermore, if $\t(\s)$ is the relativization of $\t$ to $\gg_{oo}(\s)$ and $\a(\s)$ is defined by equation \eqref{6.30} then $(\gg_{oo}(\s),\t(\s),\a(\s))$ is a Zariski-free manifold.
\end{thm}
\index{$\t(\s)$}
\subsection{A holomorphic cover of $\gg_{oo}(\s)$} 
In this subsection we shall construct a 2-to-1 holomorphic covering map from a Zariski-free  subdomain of $\s$ onto $\gg_{oo}(\s)$. For each $\gamma \in \Gamma$ and each $\tau \in \{-1,1\}^\gamma$, define $\pi\subgt :\s\subgt \to \g\subgt({\s})$ by
\[
\pi\subgt = \Phi\subgt \circ \omega\subgt^{-1}.
\]
Thus, for $w\in\s\subgt$,
\[
\pi\subgt(w)=(u,v^2, S\subgt(v^2)uS\subgt(v^2).
\]
\begin{lem}\label{lem6.20}
If $w \in \s\subgtsub{1} \cap \s\subgtsub{2}$, then $\pi\subgtsub{1}(w) = \pi\subgtsub{2} (w)$.
\end{lem}
\begin{proof}
As $w \in \s\subgtsub{1}$, there exists $(u_1,x_1) \in \calu\subgtsub{1}$ such that $w = \omega\subgtsub{1}(u_1,x_1)$. Likewise, there exists $(u_2,x_2) \in \calu\subgtsub{2}$ such that $w = \omega\subgtsub{2}(u_2,x_2)$. But
\[
 \omega\subgtsub{1}(u_1,x_1) =  \omega\subgtsub{2}(u_2,x_2)
\]
implies via equation \eqref{6.10} that
\[
 u_1+S\subgtsub{1}(x_1) = u_2+S\subgtsub{2}(x_2)
\]
and
\[
 u_1-S\subgtsub{1}(x_1) =  u_2-S\subgtsub{2}(x_2).
\]
These equations imply that $u_1 =u_2$ and $S\subgtsub{1}(x_1)=S\subgtsub{2}(x_2)$. Noting that both $S\subgtsub{1}$ and $S\subgtsub{2}$ are branches of $\sqrt x$, we also have
\[
x_1 = S\subgtsub{1}(x_1)^2 = S\subgtsub{2}(x_2)^2 = x_2.
\]
Hence
\begin{align*}
\pi\subgtsub{1}(w) &= \pi\subgtsub{1}(\omega\subgtsub{1}(u_1,x_1))\\
&=\Phi\subgtsub{1}(u_1,x_1)\\
&=(u_1,x_1,S\subgtsub{1}(x_1)u_1S\subgtsub{1}(x_1))\\
&=(u_2,x_2,S\subgtsub{2}(x_2)u_2S\subgtsub{2}(x_2))\\
&=\Phi\subgtsub{2}(u_2,x_2)\\
&=\pi\subgtsub{2}(\omega\subgtsub{2}(u_2,x_2))\\
&=\pi\subgtsub{2}(w).
\end{align*}
\end{proof}

 Armed with Lemma \ref{lem6.20} we may define $\s_{oo} \subseteq \s$ by
 \be\label{6.35}
 \s_{oo} = \bigcup_{\gamma \in \Gamma}\bigcup_{\tau \in \{-1,1\}^\gamma} \s\subgt.
 \ee
\index{$\s_{oo}$}
It is immediate that there is a globally well-defined mapping $\pi :\s_{oo} \to \gg_{oo}(\s)$, locally defined by the formula,
 \[
 \pi(w) = \pi\subgt(w) \text{ if } w\in \s\subgt.
 \]
\begin{thm}\label{thm6.20}
$\pi$ is a holomorphic covering map from the Zariski-free  domain $\s_{oo}$  onto the  Zariski-free  manifold $\gg_{oo}(\s)$. Furthermore, $\pi$ is given by the formula
\be\label{6.40}
\pi(w) = (\tfrac12 (w^1 + w^2),\tfrac14 (w^1-w^2)^2,\tfrac12 (w^1-w^2)\tfrac12 (w^1 + w^2)\tfrac12 (w^1 - w^2)).
\ee
Finally, $\pi$ is $2$-to-$1$ and $\pi^{-1}(\pi(w)) = \{w,w^f\}$
\end{thm}
\begin{proof}
By Proposition \ref{prop6.20}, $\calu\subgt(\s)$ and $\s\subgt$ are Zariski-free  domains and
\[
\omega\subgt:\calu\subgt(\s) \to \s\subgt
 \]
 is a  Zariski-free  homeomorphism. Also, by the comments following equation \eqref{6.20},
 \[
 \Phi\subgt:\calu\subgt(\s) \to \g\subgt(\s)
  \]
  is a Zariski-free  homeomorphism. Therefore,
\[
\pi|\s\subgt =  \Phi\subgt \circ \omega\subgt^{-1}
\]
is a Zariski-free  homeomorphism. This proves that $\pi$ is a local homeomorphism.

If $\alpha = \Phi\subgt^{-1}:\g\subgt(\s) \to \calu\subgt(\s)$ is a co-ordinate patch in $\a(\s)$ (cf. equation \eqref{6.30}), then
\[
\alpha \circ (\pi| \s\subgt) = \omega\subgt^{-1}
\]
and
\[
(\pi| \s\subgt) ^{-1} \circ \alpha^{-1}  = \omega\subgt
\]
are both holomorphic. Therefore $\pi$ is biholomorphic on every $\s\subgt$.

As
\[
\g\subgt(\s) = \ran \pi\subgt \subset \ran \pi,
\]
for all $\gamma \in \Gamma$ and $\tau \in \{-1,1\}^\gamma$, it follows from equation \eqref{6.20} that $\pi$ is surjective.

If $(u,x) \in \ugam$ and $w = \omega\subgt(u,x)$, then $u = \tfrac12 (w^1 + w^2)$, $S\subgt (x) = \tfrac12 (w^1 - w^2)$, and $x = \tfrac14 (w^1-w^2)^2$.  Hence,
\begin{align*}
\pi\subgt(w) &= \pi\subgt(\omega\subgt(u,x))\\
&=\Phi\subgt(u,x)\\
&=(u,x,S\subgt(x)uS\subgt(x))\\
&=(\tfrac12 (w^1 + w^2),\tfrac14 (w^1-w^2)^2,\tfrac12 (w^1-w^2)\tfrac12 (w^1 + w^2)\tfrac12 (w^1 - w^2)).
\end{align*}
This proves equation \eqref{6.40}.

Assume that $\pi(w_1) = \pi(w_2)$ for some $w_1,w_2\in \s_{oo}$. If $w_1 = \omega\subgtsub{1}(u_1,x_1)$ and $w_2 = \omega\subgtsub{2}(u_2,x_2)$ then the computation of the previous paragraph gives $u_1=u_2$, $x_1=x_2$ and
\[
S\subgtsub{1}(x_1)u_1 S \subgtsub{1}(x_1) = S\subgtsub{2}(x_2)u_2 S \subgtsub{2}(x_2).
\]
If we set $x_0 = x_1 =x_2$ and $u_0 = u_1 = u_2$, then this equation becomes equation \eqref{5.60} in the proof of Proposition  \ref{lem5.50}.  The argument in that proof shows that equation \eqref{plusorminus} holds with $x=x_0$, that is, either
\[
S\subgtsub{1}(x_0) = S\subgtsub{2}(x_0) \quad \mbox{ or } \quad S\subgtsub{1}(x_0) = -S\subgtsub{2}(x_0).
\]
In the former case $w_2 =w_1$ and in the latter case $w_2 = w_1^f$.

Finally, to show that $\pi$ is 2-to-1, observe that $w\neq w^f$ for $w\in \s\subgt$.  Otherwise, suppose $w=\omega\subgt(u,x)$ for some $(u,x)\in\calu\subgt$.  Since $w=w^f$ we have $u+S\subgt(x)=u-S\subgt(x)$, and therefore $S\subgt(x)=0$.
Hence $x=S\subgt(x)^2=0$, contrary to the fact that $(u,x)\in\calu_\gamma$.
\end{proof}

We shall use formula \eqref{6.40} to extend $\pi$ to a polynomial map $\m^2 \to \m^3$.
\subsection{The representation of symmetric functions on $\s_{oo}$}
\begin{thm}\label{thm6.30}
Let $\s$ be a free symmetric domain in $\m^2$, $\gg_{oo}(\s)$ the Zariski-free  manifold described in Theorem {\rm\ref{thm6.10}} and $\pi:\s_{oo} \to \gg_{oo}(\s)$ the covering map described in Theorem {\rm\ref{thm6.20}}. If $F_{oo}$ is a 
 holomorphic function on $\gg_{oo}(\s)$, then $f_{oo} =F_{oo} \circ \pi$ is a symmetric Zariski-free  holomorphic function on $\s_{oo}$. Conversely, if $f_{oo}$ is a symmetric Zariski-free  holomorphic function on $\s_{oo}$, then there exists a unique 
holomorphic function $F_{oo}$ on $\gg_{oo}(\s)$ such that $f_{oo} = F_{oo} \circ \pi$.
\begin{proof}
First assume that $F_{oo}$ is Zariski-freely  holomorphic on $\gg_{oo}(\s)$ and let $f_{oo} = F_{oo} \circ \pi$. As $F_{oo}$ and $\pi $ are Zariski-freely  holomorphic, so is $ f_{oo}=F_{oo} \circ \pi$. As $\pi(w^f) = \pi(w)$ for all $w \in \s_{oo}$, so also $f_{oo}(w^f) = f_{oo}(w)$ for all $w \in \s_{oo}$. Thus, $f_{oo}$ is a symmetric Zariski-freely holomorphic function.

Now assume that $f_{oo}$ is a symmetric Zariski-freely  holomorphic function on $\s_{oo}$. Attempt to define a function $F_{oo}$ on $\gg_{oo}(\s)$ by the formula
\[
F_{oo}(\pi(w)) = f_{oo}(w), \qquad w \in \s_{oo}.
\]
If $\pi(w^1) = \pi(w^2)$ and $w^1 \neq w^2$ then as $w^2 = (w^f)^1$ and $f_{oo}$ is assumed symmetric, $F_{oo}(\pi(w^1))=F_{oo}(\pi(w^2))$. This proves that $F_{oo}$ is well defined, and clearly $ f_{oo}=F_{oo} \circ \pi$. Also, as $\pi$ is surjective,
 $F_{oo}$ is defined on all of $\gg_{oo}(\s)$. Since locally on the sets $\g\subgt(\s)$, $F_{oo} = f_{oo} \circ \pi^{-1}$, a composition of Zariski-freely  holomorphic functions, $F_{oo}$ is Zariski-freely  holomorphic.
\end{proof}
\end{thm}

\subsection{The representation of symmetric functions on $\s$}
For any symmetric free domain $\s$ in $\m^2$ we define $\s_o \subseteq \s$ by
\be
\label{eqpu1}
\s_o = \bigcup_{\gamma \in \Gamma} \bigcup_{\tau \in \{-1,1\}^\gamma} \omega\subgt(\wgam) \cap \s.
\ee
\index{$\s_o$}
Proposition \ref{prop6.10} implies that $\s_o$ is a free domain and equations \eqref{6.15} and \eqref{6.35} imply that $\s_{oo} \subseteq \s_o$.
An alternative, intrinsic characterization of $\s_o$ in terms of the set $\q$ defined by equation \eqref{defQ} is given in the following proposition. 
\begin{prop}\label{prop6.30}
$\s_o = \set{w \in \s}{w^1-w^2 \in \q}$
\end{prop} 
\begin{proof}
First assume that $w=\omega\subgt(u,x) \in \s_o$, so that $\tfrac12 (w^1-w^2) = S\subgt(x)$ and $x\in D_\ga$ is nonsingular.    The spectrum of $S\subgt(x)$ is obtained by the application of $s\subgt$ to $\si(x)$, which process does not produce two eigenvalues $\la$ and $-\la$.  Hence $S\subgt(x)\in \q$, and therefore $w \in \q$. Conversely, suppose that $w \in \s$ and $\tfrac12 (w^1-w^2) \in \q$. If we set $x=\tfrac14 (w^1-w^2)^2$ and $\gamma = \sigma(x)$, then there exists $\tau \in \{-1,1\}$ such that $S\subgt(x) = \tfrac12 (w^1-w^2)$. If $u = \tfrac12 (w^1+w^2)$, then $w = S\subgt (u,x) \in \s_o$.
\end{proof}

Let $\s$ be a  symmetric free domain in $\m^2$, and define
$\s_o$ by equation \eqref{eqpu1} and $\s_{oo}$ by equation \eqref{6.35}, where $\pi$ is the polynomial map \eqref{6.40}.
Write $\gg_{oo}$ for $\gg_{oo}(\s)$, and define
\index{$\gg_{oo}$}
 $\gg_o = \pi(\s_o)$ and $\gg = \pi(\s)$.
\index{$\gg_o$}
Consider the following diagram:
\be
\label{bigpic}
\begin{tikzcd}
  &
 \arrow[swap]{ld}{f_{oo}}  \s_{oo} 
\arrow{r}{\pi} \arrow{d}{\subseteq} & \gg_{oo} \arrow{d}{\subseteq}  \arrow[dr, "F_{oo}"]  \\
\m^1 \arrow[r, leftarrow, "f_o"]& \s_{o} \arrow{r}{\pi} \arrow{d}{\subseteq} & \gg_{o} \arrow{d}{\subseteq} \arrow[r, "F_o"] & \m^1\\
& \arrow{lu}{f} \s \arrow{r}{\pi}& \gg  \arrow[swap,ur, dashed, "F"]
\end{tikzcd}
\ee
In Theorem~\ref{thm6.30} we showed that $f_{oo} : \s_{oo} \to \m^1$ is a Zariski-free holomorphic map if and only if
there exists a Zariski-free holomorphic map $F_{oo} : \gg_{oo} \to \m^1$ such that $F_{oo} \circ \pi = f_{oo}$.

Suppose $F : \gg \to \m^1$ is a free holomorphic map, by which we mean that at every point $\lambda$ of $\gg$ there is a free neighborhood $\gdel$ of $\la$ in $\m^3$ and a bounded free holomorphic function 
$g$ on $\gdel$ such that $g $ agrees with $F$  on $\gdel \cap \gg$. Then $F \circ \pi$ is a symmetric
free holomorphic function on $\s$.
\begin{ques}
\label{bigq}
Is the converse true?
\end{ques}

In this generality, the answer to Question \ref{bigq} is no, as the following example, suggested to us by
James Pascoe, shows.

\begin{exam}
{\rm
Let $\s$ be the 
 {\em nc-bidisc}, that is, the set of pairs of strict contractions in $\m^2$:
\[
\s \ = \ \{ w \in \m^2 : \| w ^1 \|, \| w^2 \| < 1 \}.
\]
Let 
\begin{align*}
f(w) \ &=\  (w^1 - w^2) (w^1 + w^2)^2 (w^1 - w^2) \\
	&= 16 vu^2 v.
\end{align*}
The natural choice for an $F$ such that $f=F\circ\pi$ is $F(z) =16 z^3 (z^2)^{-1} z^3$, but this function is not freely holomorphic
on a neighborhood of $0$.

Indeed, there is no free holomorphic function $F$ defined on a neighborhood $\gdel$ of $0$  in $\m^3$
such that $F \circ \pi = f$. For suppose that $F$ is such a function. Replacing $\delta(z)$
by $t( \delta(z) - \delta(0))$, where $ t \geq  1/(1 - \| \delta(0) \|)$, we can assume that $\delta (0) = 0$.
By \cite{amfree}, $F(z)$ can be represented by  a convergent
series  on $\gdel$  whose terms are non-commutative polynomials in the entries $\delta_{ij}(z)$.
If $F \circ \pi = f$, then
\be
\label{eqpu21}
F(u, v^2, uvu) \ = \ 16vu^2 v .
\ee
Expand $F$ in a power series in $\d$, and group terms by homogeneity.
Then the left-hand side of equation \eqref{eqpu21} is a linear combination of 
\[
u^2 v^2, u v^2 u, v^2 u^2 .
\]
No linear combination of these three elements is equal to $vu^2v$.

One can show more: there is no function $F$ at all satisfying $F \circ \pi = f$ on $\s$.
We shall show this by giving two points $w$ and $W$ in $\s\cap\m_4^2$ that are identified by $\pi$ but not by $f$.
Let 
\[
2u = w^1 + w^2, \quad  2U = W^1 + W^2, \quad 2v = w^1 - w^2, \quad 2V = W^1 - W^2.
\]
Let $r$ be a small positive number. 
Let
\[
v \ = \ 
\begin{bmatrix}
\begin{bmatrix}
0 & r \\
0 & 0
\end{bmatrix} 
& 0\\
0 &
\begin{bmatrix}
0 & r \\
0 & 0
\end{bmatrix} 
\end{bmatrix} , \quad
V \ = \
\begin{bmatrix}
\begin{bmatrix}
0 & -r \\
0 & 0
\end{bmatrix} 
& 0\\
0 &
\begin{bmatrix}
0 & r \\
0 & 0
\end{bmatrix} 
\end{bmatrix} 
\]
Let $u = U$ have the $(2,3)$ and $(4,1)$ entries equal to $0$, but 
the $(2,3)$ element of their square not equal to $0$.
Then 
$
\pi(w) = (u, v^2, vuv) = \pi(W),
$ but
$$
f(w) = 16vu^2 v \neq 16V U^2 V = f(W),
$$ 
since the $(1,4)$ entry in $vu^2v$ is $r^2$ times the $(2,3)$ entry of
$u^2$, and in $VU^2V$ it is the negative of this quantity.
}
\end{exam}

The fact that $v$ is not invertible is crucial in the example.
\begin{lem}
\label{lempu3}
Let $\s$ be a  symmetric freely open set in $\m^2$ and $f$  a symmetric free holomorphic function on $\s$.
\begin{enumerate}[\rm (1)]
\item  $f$ can be approximated locally uniformly on $\s$ in the free topology by a sequence of symmetric free polynomials;
\item Suppose that {\rm (i)} $w_1, w_2 \in \s$ 
{\rm (ii)} $\pi(w_1) = s^{-1} \pi(w_2) s$ for some $s\in\inv$ {\rm (iii)} $w_1^1 - w_1^2$ is invertible and {\rm (iv)} $w_1 \oplus w_2 \oplus w_1^f \oplus w_2^f \in \s$;  then $f(w_1) = s^{-1} f(w_2) s$.
\end{enumerate}
\end{lem}
\begin{proof}  (i)  Consider any point $w_0 \in\s$.
Since $f$ is freely holomorphic on $\s$, by Theorem~\ref{thm2.10} there is a basic free neighborhood $\gdel$ of $w_0$ in $\s$
and a sequence of free polynomials $p_n$ that approximates $f$ uniformly on $\gdel$.
 Replacing $\delta(w) $ by $\delta(w) \oplus \delta(w^f)$ if necessary, we can assume that $\gdel$
is symmetric, and on replacing $p_n(w) $ by $\tfrac 12 (p_n(w) + p_n(w^f))$, we can assume that $p_n$ is symmetric on $\gdel$.\\
(ii)  Let 
\[
u_1 =\half (w_1^1 + w_1^2),\quad  v_1 = \half(w_1^1 - w_1^2), \quad u_2 = \half(w_2^1 + w_2^2), \quad v_2 = \half(w_2^1 - w_2^2).
\] 
Every $p_n$, being symmetric, can be written as a finite linear combination of terms $1, u, $ and $v u^j v$
for $j = 0,1,2, \dots.$ 

By hypothesis, $\pi(w_1) = s^{-1} \pi(w_2) s$, which means that $u_1 = s^{-1} u_2s$ and
 $v_1 u_1^j v_1 =s^{-1} v_2 u_2^j v_2s $ for
$j = 0,1$. But for $j \geq 2$, we have 
\[
v_1 u_1^j v_1 \ = \  v_1 u_1^{j-1} v_1 (v_1)^{-2} v_1 u_1 v_1 ,
\]
and so by induction we deduce that $p_n(w_1) = s^{-1}p_n(w_2)s$.
Since $p_n(w_1)$ converges to $f(w_1)$ and $p_n(w_2)$ converges to $f(w_2)$, it follows that 
 $f(w_1) = s^{-1} f(w_2) s$.
\end{proof}

\begin{lem}
\label{lempu1}
The set of points $w \in \m^2$ such that $\pi^{-1} (\pi (w))$ has cardinality $2$ is dense
in the finitely open topology.
\end{lem}
\begin{proof}
Let $2u = w^1 + w^2$ and $2v = w^1 - w^2$.
The set of $w \in \m_n^2$ for which $v$ has $n$ distinct non-zero eigenvalues is dense,
and the subset where furthermore $u$ has no non-zero entries when expressed with respect to the basis
of eigenvectors of $v$ is still dense.

For such a $w$, if $\pi(w_1) = \pi(w)$, then $u = u_1,\  (v)^2 = (v_1)^2$ and $ v u v = v_1 u v_1$.
The second equation says $v_1$ has the same eigenvectors as $v$, with each eigenvalue $\pm$ the corresponding
eigenvalue for $v$. If the $i^{\rm th}$ is positive and the $j^{th}$ is negative, then the 
$(i,j)$ entry of $vuv$ is minus the $(i,j)$ entry of $v_1 u v_1$, and is non-zero since $u_{ij}$ is non-zero.

Thus $v_1 = \pm v$, and hence $w_1$ is either $w$ or $w^f$.
\end{proof}


\subsection{\Anc\ functions}
\label{subsecanc}

We shall call a set $D \subset \m^d$ {\em hereditary}
\index{set!hereditary}
 if, whenever $x \oplus y \in D$, then
$x$ and $y$ are in $D$. 

\begin{prop}\label{freearehered}
Freely open sets are hereditary.
\end{prop}
\begin{proof}
Let $D$ be a freely open set.  Consider $z,w$ such that $z\oplus w \in D$.  Then there is a matrix $\de$ of free polynomials such that
$z\oplus w \in B_\de \subseteq D$.   Since $\de(z\oplus w) = \de(z) \oplus \de(w)$, it follows that $z,w \in B_\de$, and hence
$z,w \in D$.
\end{proof}
Zariski-freely open sets, however, need not be 
hereditary.

\begin{defin}
\label{defpub2}
Let $D$ be a subset of $\m^d$ and $f$ be a mapping from $D$ to $\m$.
We say that $f$ is {\em \anc}\  if $f$ is a graded function
and

{\rm (i)} if $x, s^{-1} x s \ \in \ D$, then $f( s^{-1} x s) =  s^{-1} f(x) s$;

{\rm (ii)}  there exists a graded function $\hat{f}$ defined on the set
\be\label{defDhat}
\hat D \df \{y\in \m^d: \mbox{ there exists } x\in D \mbox{ such that } x\oplus y \in D\}
\ee
such that, for all $x\in D$ such that $x \oplus y\in D$, 
\be
\label{eqpu83}
f
\left(
\begin{bmatrix}
x& 0 \\
0 & y
\end{bmatrix} 
\right)
\ = \
\begin{bmatrix}
f(x)& 0 \\
0 & \hat{f}(y)
\end{bmatrix} .
\ee
\end{defin}
\index{function!\anc}

\begin{exam} \label{exDhat}\rm  (1) If $D=\{3\oplus 2\oplus 1\}$ then $\hat D$ is empty and the \anc\ functions on $D$ are the functions $f$ on $D$ given by
\[
f(3\oplus 2\oplus 1) =A
\]
for some diagonal matrix $A\in\m^1_3$.  The fact that $A$ is diagonal ensures that condition (i) of Definition \ref{defpub2} is satisfied,
while condition (ii) is vacuous.

(2)  If $D=\{3\oplus 2\oplus 1,3 \}$ then $\hat D=\{2\oplus 1\}$ and the \anc\ functions on $D$ are the functions $f$ on $D$ given by
\[
f(3\oplus 2\oplus 1) =a\oplus b\oplus c, \qquad f(3)=a
\]
for some $a,b,c \in \c$.  The unique graded function $\hat f$ on $\hat D$ such that condition (ii) holds is $\hat f(2\oplus 1) = b\oplus c$.
\end{exam}

The following properties of \anc\ functions are straightforward to verify.
\begin{prop}\label{ancprops}
Let $f$ be a \anc\ function on $D\subset \m^d$ and let $D_1\subseteq D$.
\begin{enumerate}[\rm(1)]
\item  There is a {\em unique} graded function $\hat f$ on $\hat D$ such that equation \eqref{eqpu83} holds.
\item If $\|f(x)\| \leq M$ for all $x\in D$ then $\|\hat f(y)\|\leq M$ for all $y\in\hat D$.
\item  $\hat f$ agrees with $f$ on $D \cap \hat D$.
\item   $\hat D_1 \subseteq \hat D$.
\item   If $f_1=f| D_1$ then $f_1$ is \anc\ and $\hat f_1=\hat f|\hat D_1$.
\end{enumerate}
\end{prop}
\begin{remark}\label{ancfree} \rm
A bounded \anc\ function $f$ on a free domain $U$ is freely holomorphic.  For any $x\in U$ we may choose a basic
neighborhood $B_\de$ of $x$ contained in $U$; then $f$ is \anc\ on $B_\de$, hence is nc on $B_\de$.  Thus $f$ is freely locally nc.

On the other hand, 
we do not know if
 a freely holomorphic function on a free domain $U$ must be \anc, since being \anc\ is not a local property.   Such a function does preserve direct sums and similarities on
basic free neighborhoods, but we do not assert that it preserves similarities on all of $U$.
\end{remark}

On a hereditary set $D$,  a function  $f$ is  \anc\ if and only if it preserves similarities 
and satisfies $f(x \oplus y) = f(x) \oplus f(y)$ whenever $x \oplus y \in D$. If $D$ is both nc and hereditary,
then a \anc\ function on $D$ is the same as an nc function.

\section{Nc Waring-Lagrange theorems}
\label{secj}

Any polynomial in $d$ commuting variables can be expressed as a polynomial in the $d$ elementary symmetric functions in the variables.
In an interesting historical survey \cite[pp. 364-365]{Funk}, H. G. Funkhauser attributes this fundamental fact about symmetric functions to Edward Waring, Lucasian Professor at Cambridge \cite{War} in 1770 and independently to Joseph Lagrange \cite{Lag} in 1798.  Lagrange's account is later, but is somewhat more explicit, so we shall call the statement the {\em Waring-Lagrange theorem}.  It is likely that Euler had the result around the same time.

We now come to our nc version of the theorem, which asserts (to oversimplify somewhat) that symmetric nc functions in 2 variables
can be factored through the map $\pi$ given by the definition \eqref{defpi}.  The precise statement is as follows.

In the theorem  we do not assume that $\s$ is an nc set.
\begin{thm}
\label{thm9.1}
Let $\s$ be a symmetric free  domain in $\m^2$. Assume that,  for all $w_1, w_2  \in \s$,
\begin{align}
\label{assump}
\mbox{ if  } \pi(w_1)  \mbox{ is similar  to } \pi(w_2)
	 \mbox{ then } w_1 \oplus w_2 \in \s.
\end{align}
Define
$\s_o$ by equation \eqref{eqpu1} and $\s_{oo}$ by equation \eqref{6.35},
and let $\gg_{oo}$ be the Zariski-free manifold $\gg_{oo}(\s)$ defined
\index{$\gg_{oo}$}
in equation \eqref{6.20}. Let $\gg_o = \pi(\s_o)$.
\index{$\gg_o$}
There are canonical bijections between the following five sets of graded functions.

{\rm (i)} Bounded  symmetric \anc\ functions $f$ defined on $\s$.

{\rm (ii)} Bounded symmetric \anc\ functions $f_o$ defined on $\s_o$.


{\rm (iii)}  Bounded symmetric Zariski-freely holomorphic functions $f_{oo}$ defined on $\s_{oo}$ that are \anc.

{\rm (iv)} Bounded holomorphic functions $F_{oo}$ defined on the Zariski-free manifold $\gg_{oo}$
that are \anc\  (when $\gg_{oo}$ is viewed as a subset of $\m^3$).

{\rm (v)} Bounded  functions $F_o$ defined on $\gg_o$  that are \anc.

Moreover, if functions $f,f_o,f_{oo},F_o$ and $F_{oo}$ correspond under these canonical bijections, then the following diagram commutes when the two copies of $\m^1$ are identified.
\be
\label{bigpic2}
\begin{tikzcd}
  &
 \arrow[swap]{ld}{f_{oo}}  \s_{oo} 
\arrow{r}{\pi} \arrow{d}{\subseteq} & \gg_{oo} \arrow{d}{\subseteq}  \arrow[dr, "F_{oo}"]
\\
\m^1 \arrow[r, leftarrow, "f_o"]& \s_{o} \arrow{r}{\pi} \arrow{d}{\subseteq} & \gg_{o} 
\arrow[r, "F_o"] & \m^1\\
& \arrow{lu}{f} \s 
\end{tikzcd}
\ee
\end{thm}

\begin{proof}
Starting with $f$, one can define $f_o$ and $f_{oo}$ by restriction.    By Proposition \ref{ancprops}, $f_o$ and $f_{oo}$ are \anc.  By Remark \ref{ancfree}, $f$ is freely holomorphic, and so $f_{oo}$ is \zfly holomorphic, being the restriction of the freely holomorphic function $f$ to a \zfly open set.

($f_{o} \rightarrow f$) First note that at each level $n$, the set $\s \setminus \s_o$ is thin, in the sense of Subsection \ref{subsec7.1}.
Indeed, to be in $\s \setminus \s_o$, by Proposition \ref{prop6.30},
we must have $w^1-w^2 \notin \q$. A matrix $M$ is not in $\q$ if and only if
$p_1(\lambda) = \det(\lambda - M)$ and $p_2(\lambda) = \det(\lambda + M)$ have a common zero.
This happens if and only if their resultant, which is a polynomial in the entries of $M$, vanishes.
So $\s \setminus \s_o$ is the intersection of $\s$ with the zero set of a non-constant polynomial, hence is thin.

Since $f_o$ is bounded and $\s \setminus \s_o$ is thin, the function $f_o$  has an extension by continuity at each level $n$ to a bounded function $f$
on $\s$.   By continuity, the function $f$ is  symmetric
 and preserves direct sums, in the sense that if $w\oplus y\in\s$ then $w,y\in \s$ (since $\s$ is a free domain) and $f(w\oplus y)=f(w)\oplus f(y)$.

Suppose $w = s^{-1} y s$
for some $w, y \in \s$ and $s\in\inv$.  Let $w_n \in  \s_o$ converge to $w$, and let $y_n := s w_n s^{-1}$.
For $n$ large enough, $y_n \in \s$, and since $w_n$ is in $\s_o$, so is $y_n$ by
 Proposition~\ref{prop6.30}.
Then $f_o(w_n) = s^{-1} f_o (y_n) s$, and so in the limit we obtain $f(w) =  s^{-1} f (y) s$.

($f_{oo} \rightarrow f_{o}$)
Starting with $f_{oo}$, construct $f_{o}$ as follows.

First show that $\s_o\subseteq \widehat{\s_{oo}}$.
Consider $w \in \s_o$ and let 
\[
u=\half(w^1+w^2), \quad v=\half(w^1-w^2),
\]
so that $v$ is nonsingular, by Proposition \ref{prop6.30}.    
Let $\gamma$ be the spectrum of $v^2$; then $0 \notin \ga$ and $(u,v^2) \in \calw_\ga$
by equation \eqref{5.32}.  Since $v$ is a square root of $v^2$, by Subsection \ref{subsec8.3} there exists
 $\tau \in \{-1,1\}^\ga$ such that $v=S_{\ga\tau}(v^2)$.  By the definition \eqref{6.10} of $\omega_{\ga\tau}$,
\[
\omega_{\ga\tau}(u,v^2)=(u+S\subgt(v^2),u-S\subgt(v^2)=(u+v,u-v)=(w^1,w^2)=w.
\]
Thus $w\in\omega\subgt(\calw_\ga)$.


Choose  $\lambda \in  \omega_{\gamma \tau} ( \ugam) $. 
Then  by Proposition \ref{propUgam}, $ \lambda \oplus w \in \omega_{\gamma \tau} ( \ugam) \subseteq \s_{oo}$.
Thus $\la, \la\oplus w \in \s_{oo}$, which is to say that $w\in \widehat{\s_{oo}}$.  Thus 
\[
\s_{oo}\subseteq \s_o\subseteq \widehat{\s_{oo}}.
\]
We claim that 
\be\label{samehats}
\widehat{\s_o}=\widehat{\s_{oo}}.
\ee
  The inclusion $\widehat{\s_o}\supseteq \widehat{\s_{oo}}$ is immediate.
Consider any $y\in \widehat{\s_o}$.  Then there exists $w \in \s_o$ such that $w\oplus y \in \s_o$.  By the above construction there exists $\ga\in\Ga$ and $\tau\in\{-1,1\}^\ga$ such that $w\oplus y\in\omega\subgt(\wgam)$.  Let $\la\in\omega\subgt(\ugam)$; then, again by Proposition \ref{propUgam}, $\la\oplus w$ and $\la\oplus w \oplus y$ belong to $\s_{oo}$.  Hence $y\in \widehat{\s_{oo}}$.

Define $f_o$ to be the restriction of $\widehat{f_{oo}}$ to $\s_o$.  By Proposition \ref{ancprops}(3), $f_o$ is an extension of $f_{oo}$.

Since $f_{oo}$ is \anc, there is a function $\widehat{f_{oo}}$ on $\widehat{\s_{oo}}$ such that
\be
\label{eqpu22}
f_{oo}
\left(
\begin{bmatrix}
\lambda& 0 \\
0 & w
\end{bmatrix} 
\right)
\ = \
\begin{bmatrix}
f_{oo}(\lambda)& 0 \\
0 & \widehat{f_{oo}}(w)
\end{bmatrix} .
\ee

As $f_{oo}$ is bounded, so is $f_o$, by Proposition \ref{ancprops}(2) and, since $f_{oo}$ is symmetric, so is $f_o$.

To show that $f_o$ is \anc\ we must exhibit a function $\widehat{f_o}$ on $\widehat{\s_o}$ such that, for all $y\in\widehat{\s_o}$ and all $w\in \s_o$ such that $w\oplus y\in\s_o$,
\be\label{rtp}
f_o\left(\bbm w&0\\0&y \ebm\right)=\bbm f_o(w) &0 \\ 0 &\widehat{f_o}(y) \ebm.
\ee
In view of equation \eqref{samehats}, we may define $\widehat{f_o}$ to be $\widehat{f_{oo}}$.
Then we must show that
\be\label{rtp2}
\widehat{f_{oo}}\left(\bbm w&0\\0&y \ebm\right)=\bbm \widehat{f_{oo}}(w) &0 \\ 0 &\widehat{f_{oo}}(y) \ebm
\ee
whenever $w, w \oplus y \ \in \s_o $.
As above there exists $\lambda \in \s_{oo}$ such that $ \lambda \oplus w \oplus y \in \s_{oo}$, which means
that $\lambda \oplus w$ is also in $\s_{oo}$.
Thus
\begin{eqnarray*}
\begin{bmatrix}
f_{oo}(\lambda)& 0 \\
0 & \widehat{f_{oo}}
\left(
\begin{bmatrix}
w& 0 \\
0 & y
\end{bmatrix} 
\right)
\end{bmatrix}
&\ = \ & f_{oo}   
\left(
\begin{bmatrix}
\lambda& 0 &0 \\
0 & w & 0 \\
0&0& y
\end{bmatrix} 
\right)     \\
&\ = \ &            
\begin{bmatrix}
 f_{oo}
\left(
\begin{bmatrix}
\lambda& 0 \\
0 & w
\end{bmatrix} 
\right)& 0 \\
0 & \widehat{f_{oo}}(y)
\end{bmatrix} \\
&=&        
\begin{bmatrix}
f_{oo}(\lambda)& 0 &0 \\
0 & \widehat{f_{oo}}(w) & 0 \\
0&0& \widehat{f_{oo}}(y)
\end{bmatrix}. 
\end{eqnarray*}
It follows that equation \eqref{rtp2} holds, as required.

To see that $f_o$ preserves similarities, we argue as follows. 
Let $w,  s^{-1} w s =y \in \s_o$. 
By assumption \eqref{assump} we have $w \oplus y \in \s$.
As $w$ and $y$ have the same spectrum,  
By Proposition~\ref{prop6.30}, $w \oplus y \ \in \ \s_o$.
Hence there exists $\lambda \in \s_{oo}$ such that
$\lambda \oplus w \oplus y$ is in $\s_{oo}$, as are $\lambda \oplus w$ and $\lambda \oplus y$.
Since $f_{oo}$ preserves similarities and
\[
\begin{bmatrix}
1& 0 \\
0 & s^{-1}
\end{bmatrix}
\begin{bmatrix}
\lambda& 0 \\
0 &w
\end{bmatrix}
\begin{bmatrix}
1& 0 \\
0 & s
\end{bmatrix}
\ = \
\begin{bmatrix}
\lambda& 0 \\
0 &  s^{-1} w s
\end{bmatrix},
\]
we find that $f_o( s^{-1} w s) =  s^{-1} f_o(w) s$.  Thus $f_o$ is \anc.

($f_{o} \rightarrow F_{o}$)
Consider any bounded symmetric \anc\ function $f_o$ on $\s_o$.  
By the previous construction, $f_o$ has a bounded symmetric \anc\ 
extension $f$ to $\s$.   By Remark \ref{ancfree}, $f$ is freely holomorphic.

We claim that if $\pi(w_1) = \pi(w_2)$ for some pair $w_1, w_2$ in $\s_o$, then $f(w_1) = f (w_2)$.
Indeed, by assumption \eqref{assump}, $w_1\oplus w_2 \in\s$.  By the symmetry of $\s$, $(w_1\oplus w_2)^f \in\s$, which is to say that $w_1^f\oplus w_2^f \in\s$.  By the symmetry of $\pi, \; \pi(w_1\oplus w_2) =\pi ((w_1\oplus w_2)^f)$.  Again by assumption \eqref{assump}, $w_1\oplus w_2\oplus w_1^f\oplus w_2^f \in \s$.  Furthermore, since $w_1\in\s_o$, it follows from Proposition~\ref{prop6.30} that $w_1^1-w_1^2$ is invertible.  We may therefore apply  Lemma~\ref{lempu3} to conclude that $f(w_1)=f(w_2)$.  Hence $f_o(w_1) = f_o (w_2)$.

Thus, at the level of set theory, we can define $F_o $ by $f_o = F_o \circ \pi$. 
To see that $F_o$ is \anc, suppose first that $\pi(w)$ and $\pi(w) \oplus \pi(y)$ are in
$\gg_o$, for some $w , w \oplus y \ \in \s_o$. Then $y \in \s_o$ also, by  Proposition~\ref{prop6.30}.
Thus
\begin{eqnarray*}
F_{o}
\left(
\begin{bmatrix}
\pi(w)& 0 \\
0 & \pi(y)
\end{bmatrix} 
\right)
&\ = \ &
f_{o}
\left(
\begin{bmatrix}
w& 0 \\
0 & y
\end{bmatrix} 
\right)\\
&=&
\left(
\begin{bmatrix}
f_o(w)& 0 \\
0 & f_o(y)
\end{bmatrix} 
\right)\\
&=&
\left(
\begin{bmatrix}
F_o(\pi(w))& 0 \\
0 & F_o(\pi(y))
\end{bmatrix} 
\right).
\end{eqnarray*}

To show that $F_o$ preserves similarities,
\black
consider similar triples $z_1$ and $z_2=sz_1s^{-1}$ in $\gg_{o}$.
Then there exist $w_1,w_2\in\s_o$ such that $z_1=\pi(w_1), z_2=\pi(w_2)$.
We have
\begin{align}\label{z2sz1s}
z_ 2&=sz_1s^{-1} \\
	&=s\pi(w_1)s^{-1}=\pi(sw_1s^{-1}). \notag
\end{align}
If it happens that $sw_1s^{-1}$ is in $\s_o$, then
\begin{align}
F_o(z_2)  &\ = \  F_o\circ\pi (sw_1s^{-1} ) \notag \\
& \ = \  f_o(sw_1s^{-1}) \notag \\
&\ = \ s f_o(w_1 )s^{-1}   \quad \mbox{ since } f_o \mbox{ is \anc} \notag \\
&\ = \ s F_o\circ\pi(w_1) s^{-1}  \notag \\
&\ = \ s F_o(z_1) s^{-1}.  \label{sympres}
\end{align}
Not knowing that $s w_1 s^{-1}\in\s_o$, we use Lemma~\ref{lempu3}.
By Assumption \eqref{assump},
$w_1 \oplus w_2 \oplus w_1^f \oplus w_2^f \in \s$, and since $w_1 \in \s_o$, 
the lemma tells us that 
$f(w_2) = s^{-1} f(w_1) s$, and so $F_o(z_2) =  s^{-1} F_o(z_1) s$, as required.

\black

($F_{o} \rightarrow f_o$)
Define $f_o$ to be $F_o \circ \pi$. Then $f_o$ is bounded, symmetric, and \anc. 

($F_o \rightarrow F_{oo}$) By restriction.

($F_{oo} \rightarrow f_{oo}$) Define $f_{oo}$ to be $F_{oo} \circ \pi$.

($F_{oo} \rightarrow F_o$)  By the composition $F_{oo}\rightarrow f_{oo}\rightarrow f_o \rightarrow F_o$.
\end{proof}

The set $\s_o$ is not closed with respect to direct sums, that is, it is not an nc set.
If $\s$ in Theorem \ref{thm9.1}  is assumed to be a symmetric freely open nc set, then assumption \eqref{assump} is automatically satisfied,
and the \anc\ functions on $\s$  are the same as the nc functions. This yields the following corollary.

\begin{cor}
\label{cor9.1}
Let $\s$ be a symmetric freely open nc set in $\m^2$.
Define
$\s_o$ by equation \eqref{eqpu1}, and let $\gg_o = \pi(\s_o)$.
Then there is a canonical bijection between the bounded symmetric nc functions  $f$  on $\s$
and the bounded \anc\ functions  $F_o$  defined on $\gg_o$  such that $F_o=(f|\s_o)\circ \pi$.  
\end{cor}

\black

If we drop the assumption that $f$ be bounded, and require only that $f$ be freely locally
bounded, then this imposes corresponding restrictions on the other functions. 
Recall that in Definition~\ref{def2.110} we defined a function $\phi$  defined on $U\setminus T$ to be locally bounded on $U$
if, for every point $z$ in $U$, there is a neighborhood $V $ of $z$ such that $\phi$ is bounded on $V \setminus T$.
The theorem becomes:

\begin{thm}
\label{thm9.2}
Let $\s$ be a symmetric free  domain in $\m^2$. Assume condition \eqref{assump}.

There are canonical bijections between the following five sets of graded functions.

{\rm (i)} Symmetric \anc\ functions $f$ that are freely holomorphic on $\s$.

{\rm (ii)} Symmetric \anc\ functions $f_o$ defined on $\s_o$ that are freely locally bounded on $\s$.


{\rm (iii)}  Symmetric Zariski-freely holomorphic functions $f_{oo}$ defined on $\s_{oo}$ that are \anc\
and freely locally bounded on $\s$.

{\rm (iv)} Holomorphic functions $F_{oo}$ defined on the Zariski-free manifold $\gg_{oo}$
that are \anc\  and have the property that for all $w$ in $\s$, there is a free neighborhood $U$ of $w$
such that $F_{oo}$ is bounded on $\pi(U) \cap \gg_{oo}$.

{\rm (v)} Bounded graded functions $F_o$ defined on $\gg_o$  that are \anc\
 and have the property that for all $w$ in $\s$, there is a free neighborhood $U$ of $w$
such that $F_{o}$ is bounded on $\pi(U) \cap \gg_{o}$.

Moreover, if functions $f,f_o,f_{oo},F_o$ and $F_{oo}$ correspond under these canonical bijections, then the diagram \eqref{bigpic2} commutes when the two copies of $\m^1$ are identified.

\end{thm}

\section{Nc Newton-Girard formulae}\label{girard}
Instances of the Waring-Lagrange theorem are furnished by a series of formulae for power sums in terms of elementary symmetric functions.  Classically such formulae were first given in 1629 by Albert Girard \cite{girard}, though they were subsequently often attributed to Newton \cite{newton}.
We are only concerned with polynomials in {\em two} variables $x$ and $y$. 
When these variables commute, the Newton-Girard formulae express the power sums $p_n(x,y)\df x^n+y^n$ in terms of the
elementary symmetric functions
\[
e_1(x,y)= x+y, \qquad e_2(x,y)= xy.
\]
The first four formulae are
\begin{align*}
p_1&=e_1\\
p_2&= e_1^2-2e_2\\
p_3&= e_1^3-3e_1e_2\\
p_4&= e_1^4 - 4e_1^2e_2+2e_2^2.
\end{align*}
Further formulae are obtained from the recursion $p_{n+2}=e_1p_{n+1}-e_2p_n$.

In the case of non-commuting indeterminates $x,y$ we retain the notation $p_n$ for the $n$th power sum,
but now there is no finite set of `elementary symmetric functions' in terms of which all $p_n$ can be written as polynomials.
However, the foregoing nc Waring-Lagrange theorems tell us that it {\em is} possible to express $p_n$ as a {\em rational} expression in the variables
\be\label{albega}
\al=u, \quad \beta= v^2, \quad \ga=vuv
\ee
where
\[
u=\half(x+y), \qquad v=\half(x-y).
\]
In this section we shall show how to construct such expressions explicitly, in the spirit of Girard and Newton.  Since we are obliged to work with rational expressions, it is natural to allow the $n$ in $p_n$ to be an arbitrary integer, positive, negative or zero.
We first express $p_n$ and the antisymmetric rational function 
\[
q_n\df x^n-y^n
\]
 in terms of $u$ and $v$.
We have $p_0=2, \, q_0=0$.
For any integer $n$,
\begin{align}\label{formpn}
p_n &= xx^{n-1}+yy^{n-1}=(u+v)x^{n-1}+(u-v)y^{n-1}\notag \\
	& = u(x^{n-1}+y^{n-1})+v(x^{n-1}-y^{n-1}) \notag \\
	&=u p_{n-1}+v q_{n-1}.
\end{align}
Similarly,
\be\label{formqn}
q_n= vp_{n-1}+u q_{n-1}.
\ee
We may write equations \eqref{formpn} and \eqref{formqn} in the matrix form
\be\label{matrixpnqn}
\bpm p_n \\ q_n \epm = T \bpm p_{n-1} \\ q_{n-1} \epm  \qquad \fa n\in \mathbb{Z}
\ee
where
\be\label{defT}
T= \bbm u&v \\ v & u \ebm.
\ee
Define the free polynomial $s^n_\even$ in $u,v$ for $n\geq 0$ to be the sum of all monomials in $u,v$ of total degree $n$ and of even degree in $v$.   When regarded as a polynomial in $x,y$,  $s^n_\even$ is symmetric.  Likewise, $s^n_\odd$ is defined to be the sum of all monomials in $u,v$ of total degree $n$ and odd degree in $v$.  Thus $s^n_\odd$ is antisymmetric as a polynomial in $x,y$.

By induction, for $n\geq 1$, 
\be\label{Tn}
T^n=   \bbm s^n_\even & s^n_\odd \\ s^n_\odd & s^n_\even \ebm.
\ee
By iteration of equation \eqref{matrixpnqn},
\be\label{gotpnqn}
\bpm p_n\\q_n\epm = T^n \bpm p_0\\q_0 \epm = T^n \bpm 2 \\ 0 \epm.
\ee
From equation \eqref{Tn} we obtain, for $n\geq 0$,
\begin{align}\label{pnqn}
p_n &= 2s^n_\even\\
q_n  &= 2s^n_\odd.
\end{align}

Thus
\begin{align*}
p_1 &= 2s^1_\even= 2u = 2\al \\
p_2&= 2s^2_\even= 2(u^2+v^2)=2(\al^2+\beta) \\
p_3 &= 2s^3_\even= 2(u^3+uv^2+vuv+v^2u)= 2(\al^3+\al\beta+\ga+\beta\al) \\
p_4 &= 2s^4_\even = 2(u^4+u^2v^2+uvuv+vu^2v+uv^2u+vuvu+v^2u^2+v^4)\\
	&=2(\al^4+\al^2\beta+\al\ga+\ga\beta^{-1}\ga+\al\beta\al+\ga\al + \beta\al^2+\beta^2). 
\end{align*}
In general, any monomial in which $v$ occurs with even degree can be written as a monomial in $\al, \beta, \ga$ and $\beta^{-1}$.
Indeed, starting at one end of the monomial, replace all the initial $u$'s by $\alpha$'s.
The first $v$ must be followed by another (since the number of $v$'s is even).
If it is immediately following, replace $v^2$ by $\beta$.
If there are $k$ $u$'s between the first and second $v$'s, replace $v u^k v$ by
$(\gamma \beta^{-1} )^{k-1} \gamma$. Continue in this way until all $u$'s and $v$'s have
been replaced.

We have shown the following.

\begin{thm}\label{ncNG}
{\bf (An nc Newton-Girard theorem)}  
  For every positive integer $n$ there exists a rational function $P_n$ in three nc variables such that
$p_n=P_n \circ \pi$.    Moreover $P_n(\al,\beta,\ga)$ can be expressed as a free polynomial in $\al, \beta, \gamma$ and $\beta\inverse$.
The functions $P_{n}$ can be calculated from the equation \eqref{gotpnqn} or recursively from the relations
\[
P_{0}(\al,\beta,\ga) = 2, \qquad 
Q_{0}(\al,\beta,\ga)= 0 
\]
and,
for $n\geq 0$,
\begin{align*}
P_{n+1}(\al,\beta,\ga)&= \beta P_n +Q_n \\
Q_{n+1}(\al,\beta,\ga)&= \beta P_n +\ga\beta\inverse Q_n.
\end{align*}
\end{thm}
 The first four $P_n$ are given by
\begin{align}\label{gotP1to4}
P_1&=2\al \notag\\
P_2&=2(\al^2+\beta) \notag \\
P_3&=2(\al^3+\al\beta+\ga+\beta\al) \notag \\
P_4&=2(\al^4+\al^2\beta+\al\ga+\ga\beta^{-1}\ga+\al\beta\al+\ga\al + \beta\al^2+\beta^2).
\end{align}

Now consider sums of negative powers.  By equation \eqref{pnqn},
\be\label{pqsub-1}
\bpm p_{-1} \\q_{-1} \epm = T\inverse \bpm p_0 \\q_0 \epm.
\ee
Take inverses of both sides of the identity
\[
T=\bbm 1& 0\\ vu\inverse & 1 \ebm \bbm u & 0 \\0 & u-vu\inverse v \ebm \bbm 1& u\inverse v\\ 0&1\ebm,
\]
to get
\begin{eqnarray*}
T\inverse &\ =\ &
 \bbm 1& - u\inverse v\\ 0&1\ebm
 \bbm  u\inverse & 0 \\0 & (u-vu\inverse v)\inverse \ebm 
 \bbm 1& 0\\ - vu\inverse & 1 \ebm \\
\\
 &=&
 \bbm u\inverse + u\inverse v (u - vu\inverse v)\inverse vu\inverse &
 -u\inverse v(u-vu\inverse v)\inverse \\
 - (u-vu\inverse v)\inverse vu\inverse &  (u-vu\inverse v)\inverse \ebm .
\end{eqnarray*}
This expression can be simplified, using the identities
\begin{eqnarray*}
u\inverse v(u-vu\inverse v)\inverse & \ = \ &  (u - v u\inverse v)\inverse v u\inverse \\
u\inverse + u\inverse v (u - vu\inverse v)\inverse vu\inverse &=&  (u-vu\inverse v)\inverse \\
 -u\inverse v(u-vu\inverse v)\inverse & = &  (v-uv\inverse u)\inverse ,
 \end{eqnarray*}
 to obtain the formula
\be\label{formTinv}
T\inverse = \bbm (u-vu\inverse v)\inverse & (v-uv\inverse u)\inverse \\ (v-uv\inverse u)\inverse & (u-vu\inverse v)\inverse \ebm.
\ee
From this formula and equation \eqref{pqsub-1} we deduce that
\[
\bpm p_{-1} \\q_{-1} \epm = T\inverse \bpm 2\\0 \epm = 2\bpm (u-vu\inverse v)\inverse \\ (v-u v\inverse u)\inverse \epm.
\]
Hence
\be\label{getPsub-1}
p_{-1}= 2 (u-vu\inverse v)\inverse = 2(\al - \beta \ga\inverse \beta)\inverse.
\ee
Although $q_{-1}$ is not expressible as a rational function of $\al, \beta, \ga$ the product $vq_{-1}$ {\em is}.  
\be
vq_{-1} \ = \ - 
\label{getQsub}
\beta \gamma\inverse \beta (\alpha  - \beta \gamma\inverse \beta)\inverse
\ee
This suggests the following statement.
\begin{lem}\label{vqsub-1}
For every positive integer $n$
\[
T^{-n} = \bbm f_n(u,v) & g_n(u,v) \\ g_n(u,v) & f_n(u,v) \ebm
\]
where $f_n, g_n$ are rational functions such that $f_n(u,v)$ and $vg_n(u,v)$ are expressible as rational functions of $\al,\beta,\ga$.
\end{lem}
\begin{proof}
The statement is true when $n=1$, as is easily seen from equations
\eqref{getPsub-1} and \eqref{getQsub}.
  Suppose it true for some $n\geq 1$.
From the relation $T^{-(n+1)}=T\inverse T^{-n}$ we have the equations
\begin{align}\label{recursion}
f_{n+1}(u,v)&= (u-vu\inverse v)\inverse f_n(u,v) + (v-uv\inverse u)\inverse g_n(u,v), \notag\\
g_{n+1}(u,v)&= (v-uv\inverse u)\inverse g_n(u,v) + (u-vu\inverse v)\inverse f_n(u,v).
\end{align}
The assertion follows by induction and the relations
\begin{align*}
f_{n+1}&= (u-vu\inverse v)\inverse f_n + (v-uv\inverse u)\inverse v\inverse vg_n\\
	&=  (\al -\beta\ga\inverse\beta)\inverse f_n + (\beta -\ga \beta\inverse\al)\inverse vg_n,\\
vg_{n+1} &= v(u-vu\inverse v)\inverse g_n+ v(v-uv\inverse u)\inverse f_n \\
	&= (vuv\inverse -v^2u\inverse)\inverse vg_n +(1-uv\inverse u v\inverse)\inverse f_n \\
	&= (\ga\beta\inverse-\beta \al\inverse)\inverse vg_n + (1-\al\beta\inverse\ga\beta\inverse)\inverse f_n.
\end{align*}
\end{proof}
\begin{thm}\label{NGneg}  {\bf (A Newton-Girard theorem for negative powers)}  
  For every non-negative integer $n$ there exists a rational function $P_{-n}$ in three non-commuting variables such that
$p_{-n}=P_{-n} \circ \pi$.  The functions $P_{-n}$ are given recursively by the formulae
\[
P_{0}(\al,\beta,\ga) = 2, \qquad 
Q_{0}(\al,\beta,\ga)= 0 
\]
and
\begin{align}\label{recur}
P_{-(n+1)}(\al,\beta,\ga)&= (\al -\beta\ga\inverse \beta)\inverse P_{-n} + (\beta-\ga\beta\inverse\al)\inverse Q_{-n} \notag \\
Q_{-(n+1)}(\al,\beta,\ga)&=\beta(\beta-\al\beta\inverse\ga)\inverse P_{-n}     +\beta(\ga-\beta \al\inverse \beta)\inverse Q_{-n}
\end{align}
for $n\geq 0$.
\end{thm}
\begin{proof}
Replace $n$ by $-n$ in equation \eqref{gotpnqn} to obtain
\[
\bpm p_{-n}\\q_{-n}\epm = 2T^{-n} \bpm 1 \\ 0 \epm.
\]
Thus
\[
p_{-n} = 2 f_n(u,v)
\]
in the notation of Lemma  \ref{vqsub-1}.  By that lemma,  $p_{-n}$ is expressible as a rational function of the variables $\al,\beta, \ga$.
\end{proof}

One can readily calculate the first two $P_{-n}$
 from these formulae; we expect that $P_{-2}$ can be further simplified.
\begin{align}
P_{-1}(\al,\beta,\ga) &= 2(\al - \beta \ga\inverse\beta)\inverse  \label{Psub-1}\\
P_{-2}(\al,\beta,\ga) &= 2\left( \al^2+\beta -   \al\beta(\beta\inverse\ga+\ga\inverse\beta^2)\inverse \right. \notag \\
	&\hspace*{1cm}  -   \ga (\ga\inverse\beta\ga+\beta^2)\inverse\ga -    \al\beta(\ga\beta\inverse\ga+\beta^2)\inverse \beta\al \notag\\
	&\hspace*{1cm}\left.-    (\beta\inverse\ga\beta\inverse+\beta\ga\inverse)\inverse\al  \right)\inverse. \label{Psub-2}
\end{align}

There are some minor subtleties concerning the interpretation of the foregoing Newton-Girard formulae for $p_n$.  When $n\ge 0$, since $P_n$ is a free polynomial in $\al,\beta, \ga, \beta\inverse$, the statement
\[
x^n+y^n=P_n(\al,\beta,\ga)
\]
is meaningful and valid whenever $x$ and $y$ are matrices of the same order such that $x-y$ is nonsingular.   It can also be interpreted as an identity in the {\em free field}, which is the smallest universal division ring containing the ring of free polynomials in $x,\ y$.  When $n<0$ the issue is less immediate, since then the structure of $P_n$ is more complicated.  In this paper we are concerned with functions of tuples of matrices.  In this context,
a noncommutative  rational expression is called {\em non-degenerate} if its
domain in $\m^d$ is non-empty. The domain will always be Zariski open 
at every level $n$
(restrictions on the domain come about when there is an inverse in the expression,
as whatever needs to be inverted must be non-singular).
 Different nondegenerate noncommutative  rational
expressions may have different domains where
they can be evaluated, but agree on the intersection of these domains. Such expressions are called equivalent,
and a noncommutative  rational function is formally an equivalence class of nondegenerate noncommutative  rational expressions. See for example \cite{vo17} for a discussion. For the expressions $P_{-n}$ it is easy to see that
they are non-degenerate; from equation \eqref{recur} we see that the functions can be evaluated as long as all four of the expressions
\[
 \al -\beta\ga\inverse \beta, \quad  \beta-\ga\beta\inverse\al, \quad
\beta-\al\beta\inverse\ga, \quad \ga-\beta \al\inverse \beta
\]
are invertible.
In particular, choice of
 $x,y$  as
the scalar matrices $4$ and $2$ respectively gives the values $\al=3, \ \beta=1,\ \ga=3$, and one finds that the above four expressions do evaluate to invertible matrices.

It is interesting to compare the identities in Theorems \ref {ncNG} and \ref{NGneg}, thought of as equations in the algebra of rational functions in $x$ and $y$, with the statements about nc functions contained in our main theorems in Sections \ref{UniversalDomain} and \ref{SymmetricFunctions}.
When $n\geq 0$ the symmetric free polynomial $p_n$ is freely holomorphic on $\m^2$.  Theorem \ref{thment} applies to yield a holomorphic function $P_n$ on the \zf manifold $\g$, having a certain local boundedness property and satisfying $p_n = P_n \circ \pi$ on a suitable subset of $\m^2$.   When $n<0$ we must take the domain $\s$ of $p_n$ to be the set $\{(x,y)\in\m^2: x\in\inv, y\in\inv\}$.    Theorem \ref{thment} no longer applies, so we appeal to Theorem \ref{thm9.1}.  Notice that here $\s$ is an nc set, so that Assumption \eqref{assump} is automatically satisfied.   We again deduce that there is a holomorphic function $P_n$, this time on the \zf manifold $\gg_{oo}(\s)$, satisfying a version of the relation $p_n=P_n\circ\pi$.


\bibliography{references_uniform}
\printindex
\end{document}